\documentclass[paper=a4, fontsize=12pt]{article} 
\usepackage[english]{babel} 
\usepackage{amsmath,amsfonts,amsthm, amssymb} 
\usepackage[margin=1in]{geometry}
\usepackage{graphicx}
\usepackage{float}
\usepackage{wrapfig}
\usepackage{graphicx}
\usepackage[]{xcolor}
\usepackage{bm}
\usepackage{tikz}
\usepackage{enumitem}

\usepackage{cite}

\usetikzlibrary{arrows,decorations.pathmorphing,backgrounds,positioning,fit,petri, arrows.meta, matrix, calc, decorations.markings}
\tikzset{>=latex}

\def\fe{\ensuremath\mathcal{E}}

\def\int{\ensuremath\text{Int}}

\def\ecc{\ensuremath\text{Ecc}}
\def\qcc{\ensuremath\text{Qcc}}

\theoremstyle{definition}
\theoremstyle{plain}
\newtheorem{thm}{Theorem}
\newtheorem{cor}[thm]{Corollary}
\newtheorem{lem}[thm]{Lemma}
\newtheorem{conj}[thm]{Conjecture}
\newtheorem{obs}[thm]{Observation}

\title{Maximal planar graphs that embed as centers}
\author{Brandon Du Preez \\
	brandon.dupreez@uct.ac.za\\
	Laboratory for Discrete Mathematics and Theoretical Computer Science\\
	Department of Mathematics and Applied Mathematics\\
	University of Cape Town}

\begin{document}
	\maketitle
	
	\begin{abstract}
		\noindent A maximal planar graph is a graph which can be embedded in the plane such that every face of the graph is a triangle. 
		The center of a graph is the subgraph induced by the vertices of minimum eccentricity.
		We introduce the notion of \textit{quasi-eccentric} vertices, and use this to characterize maximal planar graphs that are the center of some planar graph.
		We also present some easier to check only necessary / only sufficient conditions for planar and maximal planar graphs to be the center of a planar graph.
		Finally, we use the aforementioned characterization to prove that all maximal planar graphs of order at most 8 are the center of some planar graph --- and this bound is sharp.
	\end{abstract}

\section{Definitions and introduction}
A graph is \textbf{maximal planar} if it is planar, but the addition of any edge destroys planarity. 
An embedding of a maximal planar graph into the plane is a \textbf{maximal plane graph}. 
A plane graph of order at least three is maximal plane if and only if every face of the graph is bounded by a 3-cycle.

If $H$ is a path or cycle in some graph, let $\ell(H) = |E(H)|$ denote the \textbf{length} of $H$.
If $G$ is a graph, we use $V(G)$ and $E(G)$ to refer to the sets of vertices and edges of $G$, respectively.
Let $G$ and $H$ be graphs.
The \textbf{Cartesian Product} $G\times H$ is the graph with vertex set $V(G\times H) = \{(u,v): u\in V(G), v\in V(H) \}$ and edge set $E(G\times H) = \{(u,v)(u',v'): (u=u' \text{ and } vv'\in E(H)) \text{ or } (v=v' \text{ and } uu' \in E(G)) \}$.
The \textbf{union} of $G$ and $H$ is the graph $G\cup H = (V(G)\cup V(H), E(G)\cup E(H))$, and the \textbf{intersection} is the graph $G\cap H = (V(G)\cap V(H), E(G)\cap E(H))$.
If $G$ is a plane graph and $f$ is a face of $G$, then $G[f]$ denotes the graph consisting of all the edges and vertices of $G$ that lie on the boundary of $f$.
If $S$ and $T$ are sets of vertices of $G$, then $T$ \textbf{dominates} $S$ if $S\subseteq N[T]$.
If $T$ is the set of vertices on the boundary of a face $f$, we say that the face $f$ dominates $S$.

Let $G=(V,E)$ be a simple graph, let $u$ and $v$ be vertices of $G$, and let $S$ be a subset of $V$. 
For all the definitions to follow, we omit the subscript $G$ if the graph in question is unambiguous.
The \textbf{induced subgraph} $G[S]$ is the subgraph of $G$ with vertex set $S$, such that two vertices of $S$ are adjacent in $G[S]$ if and only if they are adjacent in $G$.
If the induced subgraph $G[V-S]$ is disconnected, then $S$ \textbf{separates} $G$, and we call $S$ a \textbf{separating set}.
The \textbf{distance} between $u$ and $v$ in $G$, $d_G(u,v)$, is the length of a shortest $u-v$ path in $G$.
Such a path is a $u-v$ \textbf{geodesic}.
If $A$ and $B$ are subsets of $V$, the distance between these sets is given by:
\[
d_G(A,B) = \min\{d(a,b) :a\in A, b\in B \}.
\]
We let $d_G(v,A)$ = $d_G(\{v\}, A)$.
If $H$ and $K$ are subgraphs of $G$, we use the notation $d_G(H,K)$ to refer to the distance $d_G(V(H), V(K))$.
Let $C$ be a cycle of $G$, and let $u$ and $v$ be vertices of $C$.
An edge $uv$ of $E(G)-E(C)$ is a {$\bm{k}$}\textbf{-chord} if $d_C(u,v) = k$.
The \textbf{eccentricity} of $u$ in $G$ is $e_G(u) = \text{max}\{d_G(u,x) : x\in V \}$.
The \textbf{radius} and \textbf{diameter} of $G$, denoted $\text{rad}(G)$ and $\text{diam}(G)$, are the minimum and maximum eccentricities among the vertices of $G$, respectively.
The \textbf{center} of $G$ is the subgraph induced by the vertices of minimum eccentricity.
A \textbf{peripheral vertex} is a vertex whose eccentricity is equal to the graph's diameter, and a \textbf{central vertex} is a vertex whose eccentricity is equal to the graph's radius.

The $i^{th}$ \textbf{eccentricity layer} of $G$, $\mathcal{E}_G(i)$, is the set of all vertices of $G$ with eccentricity $i$.
A subgraph $H$ of $G$ is \textbf{equi-eccentric} in $G$ if there is some integer $i$ such that $V(H)\subseteq \mathcal{E}_G(i)$.
Clearly, the center of a graph is an equi-eccentric subgraph.
A subgraph $H$ of $G$ is \textbf{isometric} if, for all pairs of vertices $u$ and $v$ in $H$, we have $d_H(u,v)=d_G(u,v)$.
If $G$ is a planar graph, and $H$ is a subgraph of $G$ which is maximal planar, then $H$ is always isometric in $G$ \cite{casablanca_centersmpg}.

\begin{lem}\textup{\cite{casablanca_centersmpg}}
	\label{lem:mpgiso}
	Every maximal planar subgraph of a planar graph is isometric.
\end{lem}

\begin{proof}
	Assume to the contrary that $H$ is a maximal planar subgraph of a planar graph $G$ and that $H$ contains vertices $u$ and $v$ with $d_G(u,v) < d_H(u,v)$, and let $P$ be a $u-v$ geodesic in $G$. 
	Because $H$ is maximal planar, it is an induced subgraph of $G$, so $P$ contains a vertex $w$ in $V(G)-V(H)$. 
	Since the vertex $w$ lies in a face $f: x,y,z$ of $H$, the path $P$ contains at least two of the vertices on the boundary of $f$, say $x$ and $y$. 
	We can thus replace the segment of $P$ from $x$ to $y$ with the edge $xy$ to obtain a shorter $u-v$ path, which yields a contradiction.
\end{proof}

It is well known that every graph $G$ is the center of some graph $H$ \cite{bms:center}.
One method to construct $H$ is as follows: add four vertices $a, b, c, d$ to $G$, and make $a, b$ adjacent to each vertex of $G$, $c$ adjacent to $a$, and $d$ adjacent to $b$ (See Figure \ref{fig:hedetniemi_const}). 
However, even if $G$ is planar, the graph $H$ constructed in \cite{bms:center} having $G$ as its center is not planar if $G$ contains any vertex of degree at least three \cite{casablanca_centersmpg}.

In fact, there exist (maximal) planar graphs which cannot be the center of any planar graph. 
For example, the graph in Figure \ref{fig_cycle_not_qef} is not contained in the center of any planar graph. 

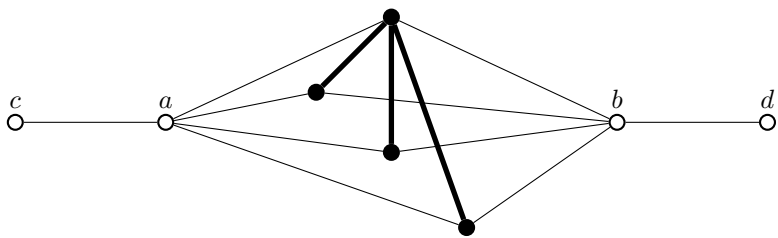
\begin{figure}[h]
	\begin{center}
		\begin{tikzpicture}
			[scale=1,inner sep=0.7mm, 
			vertex/.style={circle,thick,draw}, 
			thickedge/.style={line width=2pt}] 
			
			\node[vertex] (c) at (0,0) [label=90:{$c$}] {};
			\node[vertex] (a) at (2,0) [label=90:{$a$}] {};
			\node[vertex] (b) at (8,0) [label=90:{$b$}] {};
			\node[vertex] (d) at (10,0) [label=90:{$d$}] {};
			
			\node[vertex, fill=black] (1) at (4,0.4) {};
			\node[vertex, fill=black] (2) at (5,1.4) {};
			\node[vertex, fill=black] (3) at (6,-1.4) {};
			\node[vertex, fill=black] (4) at (5,-0.4) {};
			
			\draw (a)--(1)--(b) (a)--(2)--(b) (a)--(3)--(b) (a)--(4)--(b);
			\draw (c)--(a) (b)--(d);
			\draw[thickedge] (1)--(2)--(3) (2)--(4);
			
		\end{tikzpicture}
	\end{center}
	\caption{Given any graph $G$, the Hedetniemi construction yields a graph $H(G)$ with $G$ as its center. 
		In the example above, the vertices and edges of $G$ are bold.}
	\label{fig:hedetniemi_const}
\end{figure}

A natural starting point for investigating centers of graphs is to consider graphs which are their own centers. 
In \cite{buckley:selfcentered}, Buckley gives a survey of results and topics concerning such graphs. 
Jarry and Laugier give a proof of a strengthening of Buckley's theorem bounding the number of edges in a self-centered graph in \cite{AJAL}.

The centers of a number of graph classes, including maximal outerplanar graphs and chordal graphs, have already been described.
An \textbf{outerplanar graph} is a planar graph which can be embedded in the plane such that every vertex is on the boundary of the outer face of the graph. 
A \textbf{maximal outerplanar graph} is an outerplanar graph to which the addition of any edge results in a graph that is not outerplanar.
In \cite{prosk:centermop}, Proskurowski shows that the center of any maximal outerplanar graph is one of seven graphs, all of which are themselves maximal outerplanar (See Figure \ref{fig:mop_centers}).

\begin{figure}[h]
	\begin{center}
		\begin{tikzpicture}
			[scale=1,inner sep=0.7mm, 
			vertex/.style={circle,thick,draw}, 
			thickedge/.style={line width=2pt}] 
			
			\node[vertex] (11) at (0,0) {};
			
			\node[vertex] (21) at (1,0.5) {};
			\node[vertex] (22) at (1,-0.5) {};
			\draw (21)--(22);
			
			\node[vertex] (31) at (2,-0.5) {};
			\node[vertex] (32) at (3,-0.5) {};
			\node[vertex] (33) at (2.5,0.5) {};
			\draw (31)--(32)--(33)--(31);
			
			\node[vertex] (41) at (4,-0.5) {};
			\node[vertex] (42) at (5,-0.5) {};
			\node[vertex] (43) at (4.5,0.5) {};
			\node[vertex] (44) at (5.5,0.5) {};
			\draw (41)--(42)--(43)--(41) (43)--(44)--(42);
			
			\begin{scope}[shift={(6.5,0)}]
				\node[vertex] (1) at (0,-0.5) {};
				\node[vertex] (2) at (1,-0.5) {};
				\node[vertex] (3) at (0.5,0.5) {};
				\node[vertex] (4) at (1.5,0.5) {};
				\node[vertex] (5) at (2,-0.5) {};
				\draw (1)--(2)--(3)--(1) (3)--(4)--(2) (2)--(5)--(4);
			\end{scope}
			
			\begin{scope}[shift={(9.5,0)}]
				\node[vertex] (1) at (0,-0.5) {};
				\node[vertex] (2) at (1,-0.5) {};
				\node[vertex] (3) at (0.5,0.5) {};
				\node[vertex] (4) at (1.5,0.5) {};
				\node[vertex] (5) at (2,-0.5) {};
				\node[vertex] (6) at (2.5,0.5) {};
				\draw (1)--(2)--(3)--(1) (3)--(4)--(2) (2)--(5)--(4) (4)--(6)--(5);
			\end{scope}
			
			\begin{scope}[shift={(13,0)}]
				\node[vertex] (1) at (0,-1) {};
				\node[vertex] (2) at (1,-1) {};
				\node[vertex] (3) at (2,-1) {};
				\node[vertex] (4) at (0.5,0) {};
				\node[vertex] (5) at (1.5,0) {};
				\node[vertex] (6) at (1,1) {};
				\draw (1)--(2)--(4)--(1) (2)--(3)--(5)--(2) (4)--(5)--(6)--(4);
			\end{scope}
			
		\end{tikzpicture}
	\end{center}
	\caption{The seven possible centers of a maximal outerplanar graph.}
	\label{fig:mop_centers}
\end{figure}
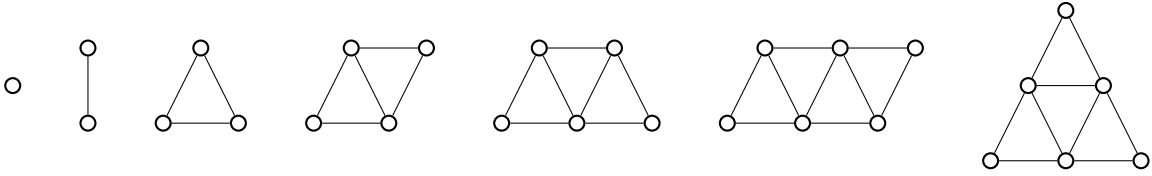

A graph is \textbf{chordal} if the only induced cycles are 3-cycles. 
Laskar and Shier showed in \cite{laskar:centerchordal} that the center of a connected chordal graph is itself a connected chordal graph. 
The center of a planar graph, or even a maximal planar graph, is not necessarily connected as Figure \ref{fig:disccenter} from \cite{casablanca_centersmpg} illustrates.

\begin{figure}[h]
	\begin{center}
		\begin{tikzpicture}
			[scale=0.4,inner sep=0.7mm, 
			vertex/.style={circle,thick,draw}, 
			thickedge/.style={line width=2pt}] 
			\node[vertex] (a1) at (0,0) [fill=white] {};
			\node[vertex] (a2) at (2,0) [fill=white] {};
			\node[vertex] (a3) at (4,0) [fill=white] {};
			\node[vertex] (a4) at (6,0) [fill=white] {};
			\node[vertex] (a5) at (8,0) [fill=white] {};
			\draw (a1)--(a2)--(a3)--(a4)--(a5);
			\node[vertex] (c1) at (4,2) [fill=black] {};
			\node[vertex] (d1) at (4,-2) [fill=black] {};
			\draw (a2)--(c1)--(a4)--(d1)--(a2) (c1)--(a3)--(d1);
			\node[vertex] (c2) at (3,3.7) [fill=black] {};
			\node[vertex] (c3) at (5,3.7) [fill=black] {};
			\node[vertex] (d2) at (3,-3.7) [fill=black] {};
			\node[vertex] (d3) at (5,-3.7) [fill=black] {};
			\draw (c1)--(c2)--(c3)--(c1) (d1)--(d2)--(d3)--(d1) (a2)--(d2)--(a1)--(c2)--(a2) (a4)--(d3)--(a5)--(c3)--(a4);
			\node[vertex] (u) at (-5,0) [fill=white] {};
			\draw (u)--(a1) (c2)--(u)--(d2) (u) .. controls (-1,3) and (4.7,5) .. (c3);
			\draw (u) .. controls (-1,-3) and (4.7,-5) .. (d3) (u) .. controls (-1,5) and (7.5,6.7) .. (a5);
		\end{tikzpicture}
	\end{center}
	\caption{A maximal planar graph with center $2K_3$. The central vertices are black.}
	\label{fig:disccenter}
\end{figure}
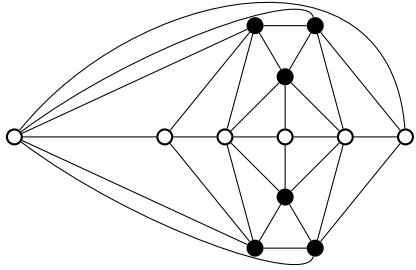

A similar problem to that of finding centers of graphs is describing the collection of eccentricities that a graph has. 
When written as an ordered sequence of positive integers, this collection is called the \textbf{eccentric sequence} of the graph.
In \cite{dankelmann:eccseqmop}, Dankelmann, Erwin, Goddard, Mukwembi and Swart characterise eccentric sequences of maximal outerplanar graphs. 

\section{Quasi-eccentricity}
Consider a (not necessarily planar) graph $G$. 
Given a vertex $v$ in $G$, we say that $u$ is an \textbf{eccentric vertex} of $v$ if $d(u,v) = e(v)$. 
Denote the set of vertices eccentric to $v$ by $\ecc(v)$. 
Given a subset $S$ of $V(G)$, we can similarly define $\ecc(S)$ as the set of vertices at maximum distance from $S$.
The eccentricity of the set $S$, $e(S)$, can be realised as the distance $d(S, \ecc(S))$.

We now introduce a similar concept.
Given a vertex $u$ and a subset $S$ of vertices of $G$, we say that $u$ is a \textbf{quasi-eccentric vertex} of $S$ in $G$ if, for any vertex $v$ of $G$, there exists a vertex $s$ in $S$ such that $d(u,s) \geq d(v,s)$. 
We denote the set of quasi-eccentric vertices of $S$ by:
\[
\qcc_G(S) = \{u \in V(G) : (\forall v \in V(G))(\exists s\in S)\text{ such that }d(u,s) \geq d(v,s) \}.
\] 
If the graph in question is clear, we omit the subscript $G$.
If $H$ is a subgraph of $G$, we use the notation $\qcc(H)$ to refer to the set $\qcc(V(H))$.
Define the \textbf{quasi-eccentricity} $q(S)$ of $S$ as $q(S) = d(S, \qcc(S))$.
Observe that quasi-eccentricity generalises eccentricity:

\begin{obs}
	Let $H$ be a graph and $S$ a set of vertices of $H$. If the vertex $u$ is an eccentric vertex of $S$, it is also a quasi-eccentric vertex of $S$.
	\label{obs_ecc_is_qcc}
\end{obs}

We illustrate the concept with an example.
Consider the path $G : v_1, \dots, v_5$ shown in Figure \ref{fig_qcc_not_ecc}.
Observe that $\ecc(S) = \{v_5\}$, while $\qcc(S) = \{v_1, v_5\}$. 
Thus $\ecc(S)$ is properly contained in $\qcc(S)$. 
Also, the eccentricity $e(S) = 2$, while the quasi-eccentricity $q(S) = 1$.
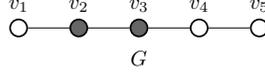
\begin{figure}[h]
	\centering
	\scalebox{0.8}{
	\begin{tikzpicture}
	[inner sep=1mm, 
	vertex/.style={circle,thick,draw},
	dvertex/.style={rectangle,thick,draw, inner sep=1.3mm}, 
	thickedge/.style={line width=1.5pt}]
	
\begin{scope}[shift={(0,0)}]
	\node[vertex] (h1) at (0,0) [label=90:{$v_1$}] {}; 
	\node[vertex, fill=black!60] (h2) at (1,0) [label=90:{$v_2$}] {}; 
	\node[vertex, fill=black!60] (h3) at (2,0) [label=90:{$v_3$}] {}; 
	\node[vertex] (h4) at (3,0) [label=90:{$v_4$}] {}; 
	\node[vertex] (h5) at (4,0) [label=90:{$v_5$}] {}; 

	\draw (h1)--(h2)--(h3)--(h4)--(h5);
	\node at (2,-0.5) {$G$}; 
\end{scope}
	
%
%
%
%
	
	\end{tikzpicture}
	}
	\caption{The path graph $G : v_1, v_2, v_3, v_4, v_5$. 
		The vertices of the set $S = \{v_2, v_3\}$ are coloured grey.}
	\label{fig_qcc_not_ecc}
\end{figure}

\section{The quasi-eccentric face criterion}
The question of whether a planar graph $H$ can be embedded into the center of some planar (or maximal planar) graph $G$ has a natural generalisation. 
We ask whether it is possible to embed $H$ into $G$ such that every vertex of $H$ has the same eccentricity in $G$, and give a necessary condition for this.

\begin{thm}[The quasi-eccentric face criterion]
	Let $H$ be a plane graph of diameter $d$, and let $\alpha \geq d$ be an integer. 
	If there exists a plane graph $G$ such that $H$ is an isometric subgraph of $G$, and for which $V(H)\subseteq \fe_G(\alpha)$, then for all vertices $u$ in $H$ that satisfy $e_H(u)< \alpha$, there exists a face $f$ of $H$ such that $u\in \qcc_H(H[f])$.
	\label{thm_qef_criterion}
\end{thm}

\begin{proof}
	Assume that there exists a vertex $u$ in $H$ with $e_H(u)<\alpha$ that is not quasi-eccentric to any face of $H$, and assume to the contrary that $H$ is an isometric subgraph of some plane graph $G$ such that $V(H)\subseteq \fe_G(\alpha)$. 
	Since the eccentricity of $u$ is less than $\alpha$ in $H$, but is exactly $\alpha$ in $G$, there is some vertex $s$ in $G-H$ with $d(u,s) = \alpha$. 
	This vertex $s$ lies in some face $f$ of $H$. 
	
	By the assumption that $u$ is not quasi-eccentric to $H[f]$ in $H$, there exists a vertex $v$ of $H$ such that $d(v,x) > d(u,x)$ for all vertices $x$ in $H[f]$. 
	Let $P: v=x_0, x_1, \dots, x_i = w, \dots, x_j = s$ be a $v-s$ geodesic in $G$, where $w$ is the last vertex of $P$ which belongs to $H[f]$. 
	Such a vertex $w$ must exist: the path $P$ starts outside of $f$ and ends in $f$. 
	Further, $v$ cannot lie in $f$ since $v$ is a vertex of $H$, and $v$ cannot lie on the boundary of $f$ because $d(v,x) > d(u,x) \geq 0$ for all vertices $x$ in $H[f]$. 
	
	Let $Q$ be a $u-w$ geodesic in $G$, and observe that $Q$ is shorter than $P[v,w]$. 
	Thus the $u-s$ path $Q\cup P[w,s]$ is strictly shorter than the $v-s$ geodesic $P$, and so the eccentricity of $v$ in $G$ is strictly greater than the eccentricity of $u$ in $G$. 
	This contradicts the assumption that $u$ and $v$ are in $\fe_G(\alpha)$, completing the proof.
\end{proof}

The next corollary follows immediately from Theorem \ref{thm_qef_criterion} and Lemma \ref{lem:mpgiso}.

\begin{cor}
	Let $H$ be a maximal plane graph of diameter $d$, and let $\alpha \geq d$ be an integer. 
	If there exists a plane graph $G$ containing $H$ as a subgraph, and for which $V(H)\subseteq \fe_G(\alpha)$, then for all vertices $u$ in $H$ that satisfy $e_H(u)< \alpha$, there exists a face $f$ of $H$ such that $u\in \qcc_H(H[f])$.
	\label{cor_qef_mpg_nec}
\end{cor}

\section{Other necessary conditions}
In this section, we explore another (much simpler to use) necessary condition for a plane graph to be an isometric, equi-eccentric subgraph of some plane graph.
We show that this necessary condition is implied by the condition of Theorem \ref{thm_qef_criterion}, but that the converse does not hold.
We first need two simple and well known lemmas, whose proofs we include for completeness.

\begin{lem}
	Let $G=(V,E)$ be a connected graph, and $S$ a separator of $G$. 
	If vertices $u$ and $v$ in $G$ are in different components of $G-S$, then $d(u,v)\geq d(u,S) + d(v,S)$.
	\label{lem_sep_different_components}
\end{lem}

\begin{proof}
	Let $P$ be a $u-v$ geodesic. 
	Because $G$ is connected and $S$ separates the vertices $u$ and $v$, there exists a vertex $s$ in $S\cap P$. 
	The geodesic $P$ can be split into two paths, $P[u,s]$ and $P[s,v]$, which have no edges in common.
	Since $P[u,s]$ is a $u-S$ path and $P[s,v]$ is an $S-v$ path, it follows that $\ell(P[u,s]) \geq d(u,S)$ and $\ell(P[s,v]) \geq d(S,v)$, and thus we obtain the following chain of inequalities:
	\[
	d(u,v) = \ell(P) = \ell(P[u,s]) + \ell(P[s,v]) \geq d(u,S) + d(v,S)
	\]
\end{proof}

\begin{lem}
	Let $G$ be a connected graph, and $S$ a connected subgraph of $G$. 
	If $u$ is a vertex of $S$, and $v$ is a vertex of $G-S$, then $d(u,v)\leq \text{diam}(S) + d(S,v)$.
	\label{lem_subgraph_diam}
\end{lem}

\begin{proof}
	Let $w$ be a vertex of $S$ such that $d(w,v) = d(S,v)$, and let $P$ be a $v-w$ geodesic. 
	Let $Q$ be a $w-u$ geodesic in $S$. 
	Since the length of $P$ is $d(S,v)$ and the length of $Q$ is at most $\text{diam}(S)$, the walk $P\cup Q$ is a $v-u$ walk of length of at most $\text{diam}(S) + d(S,v)$. 
\end{proof}

A cycle in a plane graph is a \textbf{Jordan separating cycle} if there are vertices in both its interior and exterior.
Not all separating cycles are necessarily Jordan separating, but all Jordan separating cycles are separators (See Figure \ref{fig_sepcycles}).

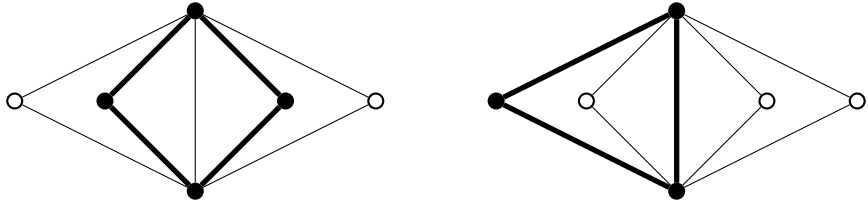
\begin{figure}[h]
	\begin{center}
		\begin{tikzpicture}
			[scale=0.8,inner sep=0.7mm, 
			vertex/.style={circle,thick,draw}, 
			thickedge/.style={line width=2pt}] 
			
			\node[vertex] (1) at (-3,0) {};
			\node[vertex, fill=black] (2) at (-1.5,0) {};
			\node[vertex, fill=black] (3) at (1.5,0) {};
			\node[vertex] (4) at (3,0) {};
			\node[vertex, fill=black] (5) at (0,1.5) {};
			\node[vertex, fill=black] (6) at (0,-1.5) {};
			
			\draw[thickedge] (5)--(3)--(6)--(2)--(5);
			\draw (5)--(4)--(6)--(1)--(5) (5)--(6);
			
			\begin{scope}[shift={(8,0)}]
				\node[vertex, fill=black] (1) at (-3,0) {};
				\node[vertex] (2) at (-1.5,0) {};
				\node[vertex] (3) at (1.5,0) {};
				\node[vertex] (4) at (3,0) {};
				\node[vertex, fill=black] (5) at (0,1.5) {};
				\node[vertex, fill=black] (6) at (0,-1.5) {}; 
				
				\draw[thickedge] (5)--(1)--(6)--(5);
				\draw (5)--(4)--(6)--(1)--(5) (5)--(6) (5)--(3)--(6)--(2)--(5);
			\end{scope}
			
		\end{tikzpicture}
		\caption{The bold cycle on the left is a separating cycle which is not a Jordan separating cycle. The bold cycle on the right is a Jordan separating cycle.}
	\end{center}
	\label{fig_sepcycles}
\end{figure}

The following necessary condition for equi-eccentric embedding, Lemma \ref{lem_cycle_cond}, shows that if a plane graph $G$ embeds isometrically into the center of some other plane graph $H$, then $G$ does not contain a short Jordan separating cycle that separates two vertices far from the cycle.

\begin{lem}
	Let $H$ be a plane graph of diameter $d$, and let $\alpha \geq d$ be an integer. 
	If there exists a plane graph $G$ such that $H$ is an isometric subgraph of $G$ and $V(H)\subseteq \fe_G(\alpha)$, then for all cycles $C$ of $H$, and all vertices $a$ and $b$ of $H$ which $C$ Jordan separates, either $d(a,C) \leq \text{diam}(C)$ or $d(b,C) \leq \text{diam}(C)$.
	\label{lem_cycle_cond}
\end{lem}

\begin{proof}
	Assume to the contrary that there exists an embedding of $H$ into $G$ as described in the hypothesis of the Lemma, but also that there exist, in $H$, vertices $a$ and $b$, and a cycle $C$, such that both $d(a,C)> \text{diam}(C)$ and $d(b,C)> \text{diam}(C)$. 
	We may assume without loss of generality that $a$ lies in the interior of $C$ and $b$ lies in the exterior of $C$. 
	Let $u$ be a vertex of $C$, and let $v$ be an eccentric vertex of $u$ in $G$. 
	Since $d(a,C)> \text{diam}(C)$, the vertex $v$ cannot be a vertex of $C$. 
	Thus $v$ must lie in either the interior or exterior of $C$. 
	Assume without loss of generality that $v$ lies in the interior of $C$. 
	Since $C$ Jordan-separates the vertices $v$ and $b$, Lemmas \ref{lem_sep_different_components} and \ref{lem_subgraph_diam} imply the following chain of inequalities:
	\[
	e(u) = d(u,v) \leq \text{diam}(C) + d(C,v) < d(b,C) + d(C,v) \leq d(b,v) \leq e(b).
	\]
	
	Thus the eccentricity of $u$ in $G$ is less than the eccentricity of $b$ in $G$, a contradiction.
\end{proof}

Lemma \ref{lem_cycle_cond_mpg}, which appears in \cite{casablanca_centersmpg}, follows immediately from Lemmas \ref{lem_cycle_cond} and \ref{lem:mpgiso}.

\begin{lem}\textup{\cite{casablanca_centersmpg}}
	Suppose that $H$ is a maximal plane graph, $C$ is a Jordan separating cycle of $H$, and that $a$ and $b$ are vertices in the interior and exterior, respectively, of $C$. 
	If $d(a,C) > \text{diam}(C)$ and $d(b,C) > \text{diam}(C)$, then $H$ is not the centre of any planar graph $G$.
	\label{lem_cycle_cond_mpg}
\end{lem}

We now show that if some plane graph $H$ satisfies the quasi-eccentric face criterion of Theorem \ref{thm_qef_criterion}, it also satisfies the condition of Lemma \ref{lem_cycle_cond}.

\begin{lem}
	Let $H$ be a plane graph of diameter $d$. 
	If for all vertices $u$ of $H$, 
	there exists a face $f$ such that $u$ is in $\qcc(H[f])$, then for all cycles $C$ of $H$, and all vertices $a$ and $b$ of $H$ which $C$ Jordan separates, either $d(a,C) \leq \text{diam}(C)$ or $d(b,C) \leq \text{diam}(C)$.
	\label{lem_qef_implies_mcc}
\end{lem}

\begin{proof}
	We prove the contrapositive. 
	Assume that $H$ is a plane graph with a cycle $C$ and vertices $a$ and $b$ such that $C$ Jordan separates $a$ and $b$, and both $d(a,C)> \text{diam}(C)$ and $d(b,C) > \text{diam}(C)$. 
	We show that there exists a non-peripheral vertex which is not quasi-eccentric to any face of $H$.
	
	First, we find a non-peripheral vertex. 
	Let $u$ be a vertex of the cycle $C$, and let $v$ be an eccentric vertex of $u$. 
	As $d(u,b) > \text{diam}(C)$, the vertex $v$ is not contained in the cycle $C$. 
	We can assume without loss of generality that $v$ is in the region of $C$ containing the vertex $b$. 
	Thus, by Lemmas \ref{lem_sep_different_components} and \ref{lem_subgraph_diam}, we obtain the following chain of inequalities:
	\[
	d(a,v) \geq d(a,C) + d(C, v) > \text{diam}(C) + d(C,v) \geq d(u,v)
	\]
	Since $v$ is eccentric to $u$, we conclude that $e(u) < d$.
	
	Now, we show that $u$ is not quasi-eccentric to any face. 
	Observe that every face of $H$ is either in the region of $C$ in which $a$ resides, or the region of $C$ in which $b$ resides. 
	Consider a face $f$ of $H$, and let $f$ be contained in the region of $C$ in which $b$ resides (it is possible that $H[f]\cap C$ is nonempty). 
	If $w$ is a vertex of $H[f]$, then by Lemmas \ref{lem_sep_different_components} and \ref{lem_subgraph_diam}, we deduce that the following inequalities must hold:
	\[
	d(a,w) \geq d(a,C) + d(C,w) > \text{diam}(C) + d(C,w) \geq d(u,w)
	\]
	Consequently, every vertex of $H[f]$ is strictly further from $a$ than it is from $u$, so $u$ is not quasi-eccentric to $H[f]$. 
	Similarly, if $f$ is a face in the region of $C$ in which $a$ lies, then $d(b,w) > d(u,w)$ for every vertex $w$ on the boundary of the face $f$. 
	In any case, the vertex $u$ is not quasi-eccentric to the face $f$.
\end{proof}

We illustrate, in Figure \ref{fig_cycle_not_qef}, a maximal plane graph $G$ which satisfies the condition of Lemma \ref{lem_cycle_cond}, but does not satisfy the quasi-eccentric face criterion of Theorem \ref{thm_qef_criterion}. 
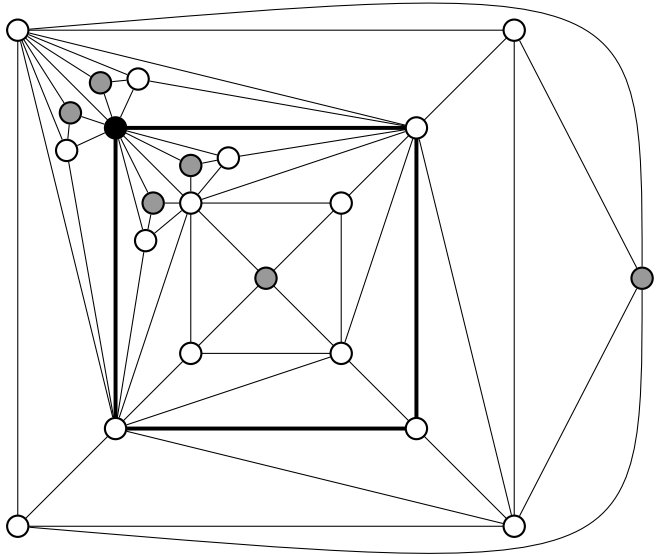
\begin{figure}[h]
	\centering
	\begin{tikzpicture}
	[inner sep=1mm, 
	vertex/.style={circle,thick,draw},
	dvertex/.style={rectangle,thick,draw, inner sep=1.3mm}, 
	thickedge/.style={line width=1.5pt}]
	
	\pgfmathsetmacro{\x}{1}
	
	\draw (1,1)--(1,-1)--(-1,-1)--(-1,1)--(1,1);
	\draw[thickedge] (2,2)--(2,-2)--(-2,-2)--(-2,2)--(2,2);
	\draw (3.3,3.3)--(3.3,-3.3)--(-3.3,-3.3)--(-3.3,3.3)--(3.3,3.3);
	\draw (-3.3,-3.3)--(3.3,3.3) (-3.3,3.3)--(3.3,-3.3);
	
	\node[vertex, fill=black!40] (0;0) at (0,0) {};
	
	\node[vertex, fill=white] (1;1) at (1*\x,1*\x) {};
	\node[vertex, fill=white] (1;-1) at (1*\x,-1*\x) {};
	\node[vertex, fill=white] (-1;-1) at (-1*\x,-1*\x) {};
	\node[vertex, fill=white] (-1;1) at (-1*\x,1*\x) {};
	
	\node[vertex, fill=white] (2;2) at (2*\x,2*\x) {};
	\node[vertex, fill=white] (2;-2) at (2*\x,-2*\x) {};
	\node[vertex, fill=white] (-2;-2) at (-2*\x,-2*\x) {};
	\node[vertex, fill=black] (-2;2) at (-2*\x,2*\x) {};
	
	\node[vertex, fill=white] (3;3) at (3.3*\x,3.3*\x) {};
	\node[vertex, fill=white] (3;-3) at (3.3*\x,-3.3*\x) {};
	\node[vertex, fill=white] (-3;-3) at (-3.3*\x,-3.3*\x) {};
	\node[vertex, fill=white] (-3;3) at (-3.3*\x,3.3*\x) {};
	
	\node[vertex, fill=black!40] (5;0) at (5,0) {};
	
	\draw (3;3)--(5;0)--(3;-3);
	\draw (-3;3) .. controls (5,4) .. (5;0);
	\draw (-3;-3) .. controls (5,-4) .. (5;0);
	
	\draw (2;2)--(-1;1)--(-2;-2)--(1;-1)--(2;2);
	\draw (2;2)--(-3;3)--(-2;-2)--(3;-3)--(2;2);
	
	\node[vertex, fill=black!40] (i1) at (-1.5,1) {};
	\node[vertex, fill=black!40] (i2) at (-1,1.5) {};
	\node[vertex, fill=white] (i1') at (-1.6,0.5) {};
	\node[vertex, fill=white] (i2') at (-0.5,1.6) {};
	
	\node[vertex, fill=black!40] (o1) at (-2.6,2.2) {};
	\node[vertex, fill=black!40] (o2) at (-2.2,2.6) {};
	\node[vertex, fill=white] (o1') at (-2.65,1.7) {};
	\node[vertex, fill=white] (o2') at (-1.7,2.65) {};
	
	\draw (-2;2)--(i1)--(-1;1) (-2;2)--(i1')--(-1;1) (i1)--(i1')--(-2;-2);
	\draw (-2;2)--(i2)--(-1;1) (-2;2)--(i2')--(-1;1) (i2)--(i2')--(2;2);
	
	\draw (-2;2)--(o1)--(-3;3) (-2;2)--(o1')--(-3;3) (o1)--(o1')--(-2;-2);
	\draw (-2;2)--(o2)--(-3;3) (-2;2)--(o2')--(-3;3) (o2)--(o2')--(2;2);
	
	\end{tikzpicture}
	\caption{The graph $G$ shows that the converse of Lemma \ref{lem_qef_implies_mcc} does no hold. Note the symmetry around the bold cycle.}
	\label{fig_cycle_not_qef}
\end{figure}
Observe that if $C$ is a separating 3-cycle in $G$ (of which there are exactly eight), then every vertex in the interior of $C$ is distance 1 from $C$. 
If $C$ is not a 3-cycle, then the criterion of Lemma \ref{lem_cycle_cond} is not broken as $G$ has diameter 4, and so no cycle $C$ of diameter 2 or greater can separate a pair of vertices $a$ and $b$ such that both $d(a,C) > 2$ and $d(b,C) > 2$. 
Thus $G$ satisfies the condition of Lemma \ref{lem_cycle_cond}. 

To see that $G$ does not satisfy the quasi-eccentric face criterion, observe that the black vertex is not quasi-eccentric to any face. 
Given any face in the interior of the bold cycle, there is a grey vertex outside the cycle which is strictly further from every vertex of the face than the black vertex is. 
The situation is similar for faces in the exterior of the bold cycle.

\section{The curious case of maximal planar graphs}
In this section, we will explore the quasi-eccentric face criterion of Theorem \ref{thm_qef_criterion} when the graph $H$, which we are embedding into some supergraph, is a maximal planar graph of order at least 4. 
We show in this case that the criterion is not only necessary, but also sufficient. 
That is to say, we show that for a maximal planar graph $H$ and an integer $\alpha \geq \text{diam}(H)$, we can embed $H$ into some (maximal) planar graph $G$ such that $V(H)\subseteq \fe_G(\alpha)$ if and only if, for all vertices $u$ in $H$, there exists a face $f$ of $H$ such that $u$ is in $\qcc_H(H[f])$.
To discuss faces of $H$, we will need to commit to an embedding of $H$ as a maximal plane graph, but the particular choice of embedding is unimportant.
By a theorem of Whitney \cite{whitney:congconn}, any two embeddings of a 3-connected planar graph have the same faces.
Diestel gives a proof of and context for this theorem in \cite{DiestelIII}.

Let $H$ be a maximal plane graph of order at least 4, and let $f$ be a face of $H$.
We can uniquely identify $f$ by the three vertices on its boundary. 
Thus we use the notation $f:x,y,z$ to indicate that $f$ is the face with vertices $x$, $y$ and $z$ on its boundary. 
Given a vertex $u$ and a face $f:x,y,z$ of $H$, the \textbf{distance vector} of $u$ relative to $H[f]$ is the ordered list: 
\[
\left(d(u,x), d(u,y), d(u,z)\right)
\]
Note that two distinct vertices $u$ and $v$ may have the same distance vector relative to $H[f]$.
If $f:x,y,z$ is a face of $H$, then the \textbf{configuration} of the set $\qcc(H[f])$ is the set of distance vectors of vertices in $\qcc(H[f])$ relative to $H[f]$. 
Symbolically, the configuration of $\qcc(H[f])$ is the set:
\[
\left\{(d(u,x), d(u,y), d(u,z)) : u \in \qcc(H[f]) \right\}.
\]

The proof that the criterion of Theorem \ref{thm_qef_criterion} is sufficient if $H$ is a maximal planar graph, to which this entire section is devoted, will unfold as follows.\\ 
Given a face $f$ of a maximal plane graph $H$, we will show that for any configuration of $\qcc_H(H[f])$, there exists a maximal plane graph $G_f$ with the following properties:
\begin{itemize}
	\item[(1)] $H$ is a subgraph of $G_f$,
	\item[(2)] Every edge and vertex which belongs to $G_f$, but not $H$, lies in $f$,
	\item[(3)] Every vertex of $\qcc_H(H[f])$ has eccentricity $\alpha$ in $G_f$,
	\item[(4)] Every vertex of $H$ has eccentricity at most $\alpha$ in $G_f$.
\end{itemize}

The proof will conclude by constructing the graph $G$ as the union over all faces $f$ of $H$ of the graphs $G_f$.

We begin by finding constraints on the configuration of the set $\qcc_H(H[f])$ for a given face $f$ of a maximal plane graph. 
The constraints we determine here are what allow us to guarantee that, for any possible configuration of $\qcc_H(H[f])$, we can construct the desired graph $G_f$.

\begin{lem}
	Let $H$ be a maximal plane graph and $f$ a face of $H$. 
	If $u$ and $v$ are vertices of $\qcc(H[f])$, then $|d(u, H[f]) - d(v, H[f])| \leq 1$.
	\label{lem_mpg_qef_diffofone}
\end{lem}

\begin{proof}
	Assume to the contrary, and without loss of generality, that $d(u, H[f]) - d(v, H[f]) \geq 2$, and let $w$ be any vertex of $H$. 
	Since all three vertices of $H[f]$ are mutually adjacent, we have that $d(w, x) \leq d(w, H[f])+1$ for all vertices $x$ in $H[f]$. 
	Thus for all vertices $x$ in $H[f]$, we have that $d(u,x) > d(v,x)$, which contradicts the fact that $v$ is quasi-eccentric to $f$.
\end{proof}

We deduce from Lemma \ref{lem_mpg_qef_diffofone} that if the quasi-eccentricity $q(H[f]) = k$, then the distance between $H[f]$ and any vertex of $\qcc(H[f])$ is either $k$ or $k+1$. 

In order to establish stronger constraints on the relationship between $f$ and $\qcc(H[f])$, we need to begin describing the configuration of $\qcc(H[f])$ in more detail. 
If $u$ is a vertex of $H$, and $f:x,y,z$ is a face of $H$, observe that $d(u,H[f])$ is the minimum of the three distances $d(u,x)$, $d(u,y)$ and $d(u,z)$.

Throughout the rest of this section, we will assume that $H$ is a maximal plane graph, and that $f:x,y,z$ is a face of $H$. 
Let $(f,k) \subseteq \qcc(H[f])$ denote the set of all quasi-eccentric vertices of $H[f]$ that are distance $k$ from $H[f]$. 
If the context makes it clear that we are referring to the face $f$, we refer to this set as $(k)$. 
For a vertex $t$ in $\{x,y,z\}$, we add a subscript $(k)_t$ to denote the subset of $(k)$ consisting of the vertices $u$ satisfying $d(u,t) = d(u,H[f]) = k$. 
We add a superscript $(k)^t$ to indicate the set of vertices $u$ in $(k)$ satisfying $d(u,t) = d(u,H[f])+1 = k+1$. 
Thus, if we say that $u$ is in $(k)^{y,z}_x$, it means that $u$ is a quasi-eccentric vertex of $f$ satisfying both $k = d(u,H[f]) = d(u,x)$ and $k+1 = d(u,y) = d(u,z)$.

\begin{figure}[h]
	\centering
	\begin{tikzpicture}
	[inner sep=1mm, 
	vertex/.style={circle,thick,draw},
	dvertex/.style={rectangle,thick,draw, inner sep=1.0mm}, 
	thickedge/.style={line width=1.3pt}]
	
	\pgfmathsetmacro{\x}{1}
	
	\filldraw[black!15] (-1,4)--(1,4)--(0,3)--(-1,4);
	
	\node[vertex] (a) at (0,6.5) [label=90:{$a$}] {};
	\node[vertex] (b) at (0,1) [label=270:{$b$}] {};
	\node[vertex] (c) at (-5,0.3) [label=270:{$c$}] {};
	\node[vertex] (d) at (5,0.3) [label=270:{$d$}] {};
	
	\node[vertex, fill=black] (x) at (-1,4) [label=135:{$x$}] {};
	\node[vertex, fill=black] (y) at (1,4) [label=45:{$y$}] {};
	\node[vertex, fill=black] (z) at (0,3) [label=90:{$z$}] {};
	
	\node[vertex] (e) at (-1,2) {};
	\node[vertex] (f) at (0,2) {};
	\node[vertex] (g) at (1,2) {};
	
	\draw[thickedge] (x)--(y)--(z)--(x);
	\draw (a)--(c)--(d)--(a) (x)--(a)--(y);
	\draw (x)--(c)--(e)--(x) (y)--(d)--(g)--(y);
	\draw (c)--(b)--(d);
	\draw (z)--(e)--(b)--(g)--(z) (z)--(f)--(b) (e)--(f)--(g);
	
	\end{tikzpicture}
	\caption{The graph $G^*$ is a maximal plane graph. The face $f^*:x,y,z$ is shaded grey.}
	\label{fig_k_notation}
\end{figure}
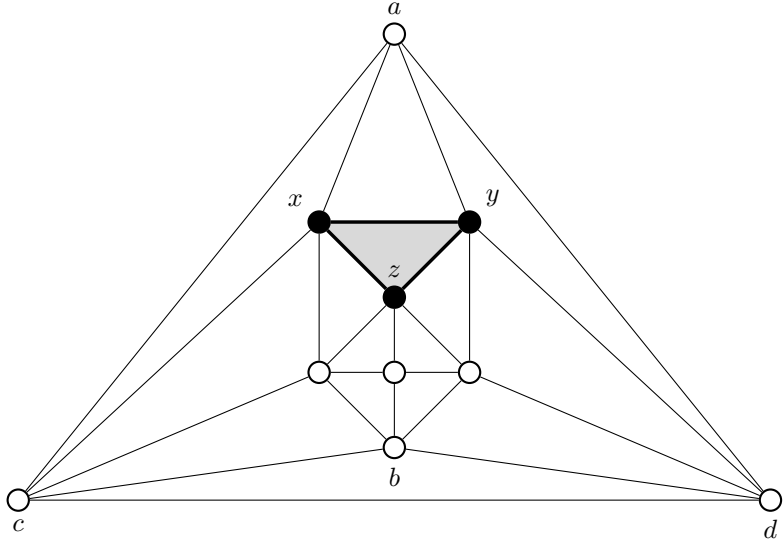

For example, consider the maximal plane graph $G^*$ in Figure \ref{fig_k_notation}. 
The face $f^*:x,y,z$ is shaded grey. 
All of the vertices $x$, $y$ and $z$ have eccentricity 2, and every vertex of $G^*-\{x,y,z\}$ is distance 2 from at least one of $x$, $y$ or $z$. 
We thus deduce that every vertex of $G^*-\{x,y,z\}$ is a quasi-eccentric vertex of $G^*[f^*]$. 
Since the distances $d(a,x) = d(a,y) = 1$, and $d(a,z) = 2$, the vertex $a$ is in $(1)^{z}_{x,y}$, and the quasi-eccentricity $q(G^*[f^*]) = 1$. 
Similarly, vertex $b$ is in $(2)_{x,y,z}$, vertex $c$ is in $(1)^{y,z}_x$, and vertex $d$ is in $(1)^{x,z}_y$.

Should $q(H[f]) = k$, there will always be a vertex in $(k)$, and by Lemma \ref{lem_mpg_qef_diffofone}, all vertices in $\qcc(H[f])$ will be vertices of $(k)$ or $(k+1)$. 
Further, if a vertex $u$ is in $(k)$, then for all vertices $t$ in $\{x,y,z\}$, we have that $d(u,t) = k$ or $d(u,t) = k+1$, since the vertices $x$, $y$ and $z$ are all mutually adjacent. 
Similarly, if $u$ is in $(k+1)$, then $d(u,t) = k+1$ or $d(u,t) = k+2$ for all $t$ in $\{x,y,z\}$.
The next series of lemmas will establish other constraints on the configuration of $\qcc(H[f])$. 
For all of these lemmas, assume that $H$ is a maximal plane graph and that $f:x,y,z$ is a face of $H$.
Unless explicitly stated otherwise, we make no assumptions about the value of $q(H[f])$.

\begin{lem}
	If $q(H[f]) = k$ and $(k)_{x,y,z}$ is non-empty, then every vertex of $\qcc(H[f])$ is in $(k)$.
	\label{lem_ak_rest}
\end{lem}

\begin{proof}
	By Lemma \ref{lem_mpg_qef_diffofone}, and since $q(H[f]) = k$, every vertex of $\qcc(H[f])$ is either distance $k$ or distance $k+1$ from $H[f]$.
	Assume to the contrary there is exists a vertex $u$ in $(k)_{x,y,z}$, and a vertex $v$ in $(k+1)$. 
	Then $d(v,t) > d(u,t)$ for all $t$ in $\{x,y,z\}$, contradicting that $u$ is quasi-eccentric to $H[f]$.
\end{proof}

\begin{lem}
	If $(k)_{x,y}^z$ is non-empty, then $(k+1)_{x,y}^z$ is empty, and vice-versa. 
	Similarly, at most one of $(k)_{x}^{y,z}$ and $(k+1)_{x}^{y,z}$ is non-empty.
	\label{lem_same_rest}
\end{lem}

\begin{proof}
	If $\qcc(H[f])$ did contain a vertex $u$ in $(k)_{x,y}^{z}$, and a vertex $v$ in $(k+1)_{x,y}^{z}$, then $v$ would be further from each of $x$, $y$ and $z$ than $u$, contradicting the quasi-eccentricity of $u$. 
	The case for $(k)_{x}^{y,z}$ and $(k+1)_{x}^{y,z}$ is similar.
\end{proof}

\begin{lem}
	If $(k+1)_x ^{y,z}$ is non-empty, then $(k)_{x}$ is empty.
	\label{lem_cforb_rest}
\end{lem}

\begin{proof}
	Assume to the contrary that $(k)_{x}$ contains a vertex $u$ and $(k+1)_{x} ^{y,z}$ contains a vertex $v$. Then $v$ is further from each of $x$, $y$ and $z$ than $u$ is. 
	This contradicts the fact that $u$ is quasi-eccentric to $H[f]$.
\end{proof}

We observe the following consequence of Lemma \ref{lem_cforb_rest}:

\begin{cor}
	If both $(k+1)_x ^{y,z}$ and $(k+1)_y ^{x,z}$ are non-empty, then $(k)_{s,t}$ is empty for any pair of distinct vertices $s$ and $t$ in $\{x,y,z\}$. 
	Further, if it is also true that $q(H[f])=k$, then $(k)_z ^{x,y}$ is non-empty.
	\label{cor_twocnob_rest}
\end{cor}

\begin{proof}
	By Lemma \ref{lem_cforb_rest}, if both $(k+1)_x ^{y,z}$ and $(k+1)_y ^{x,z}$ are non-empty, then both $(k)_x$ and $(k)_y$ are empty. 
	Since $\{s,t\}$ must contain either $x$ or $y$, the set $(k)_{s,t}$ is a subset of either $(k)_x$ or $(k)_y$. 
	Thus $(k)_{s,t}$ is empty.
	
	If $q(H[f]) = k$, then there must exist some vertex $u$ in $(k)$. 
	By Lemma \ref{lem_cforb_rest}, the vertex $u$ cannot be in $(k)_x$ or $(k)_y$, thus $u$ is in $(k)_z ^{x,z}$.
\end{proof}

As a consequence of Lemmas \ref{lem_ak_rest}, \ref{lem_same_rest} and \ref{lem_cforb_rest}, we deduce another corollary:

\begin{cor}
	If $q(H[f]) = k$, then at most two of the sets $(k+1)_x ^{y,z}$, $(k+1)_y ^{x,z}$ and $(k+1)_z ^{x,y}$ are non-empty.
	\label{cor_no_three_c_rest}
\end{cor}

\begin{proof}
	If all three sets listed in the hypothesis are non-empty, then all of the sets $(k)_x$, $(k)_y$ and $(k)_z$ are empty, by Lemma \ref{lem_cforb_rest}. 
	Thus no vertex of $\qcc(H[f])$ is distance $k$ from any of the vertices $x$, $y$ or $z$ in $H[f]$, contradicting the assumption that $q(H[f]) = k$.
\end{proof}

We are ready to begin proving that the quasi-eccentric face criterion of Theorem \ref{thm_qef_criterion} is sufficient when $H$ is a maximal plane graph. 
We show that given a face $f:x,y,z$ of $H$, and an integer $\alpha \geq \text{diam}(H)$, we can construct a plane supergraph $G_f$ such that $G_f[f]$ has, for each vertex $u$ in $\qcc_H(H[f])$, a vertex at distance $\alpha$ from $u$. 
We will further show that no vertex of $G_f[f]$ is further than $\alpha$ from any vertex of $H$.

\begin{figure}[h]
	\centering
	\begin{tikzpicture}
	[inner sep=0.7mm, 
	vertex/.style={circle,thick,draw},
	dvertex/.style={rectangle,thick,draw, inner sep=1.0mm}, 
	thickedge/.style={line width=1.3pt}]
	
	\pgfmathsetmacro{\x}{1}
	\pgfmathsetmacro{\y}{0.5}
	
	\node[vertex, fill=black!60] (s) at (0,0) {};
	
	\node[vertex, fill=black] (x0) at (90:3*\x) [label=90:{$x_0$}] {};
	\node[vertex, fill=black] (y0) at (90+120:3*\x) [label=90+120:{$y_0$}] {};
	\node[vertex, fill=black] (z0) at (90+240:3*\x) [label=90+240:{$z_0$}] {};
	
	\node[vertex] (x1) at (90:2*\x) {};
	\node[vertex] (y1) at (90+120:2*\x) {};
	\node[vertex] (z1) at (90+240:2*\x) {};
	
	\node[vertex] (xy1) at (150:0.25 + 2*\y) {};
	\node[vertex] (yz1) at (150+120:0.25 + 2*\y) {};
	\node[vertex] (zx1) at (150+240:0.25 + 2*\y) {};
	
	\node[vertex] (xy2) at (150:0.25 + 1*\y) {};
	\node[vertex] (yz2) at (150+120:0.25 + 1*\y) {};
	\node[vertex] (zx2) at (150+240:0.25 + 1*\y) {};
	
	\node[vertex, fill=black!40] (x2) at (90:1*\x) {};
	\node[vertex, fill=black!40] (y2) at (90+120:1*\x) {};
	\node[vertex, fill=black!40] (z2) at (90+240:1*\x) {};
	
	\draw[thickedge] (x0)--(y0)--(z0)--(x0);
	\draw (x1)--(y1)--(z1)--(x1);
	\draw (x2)--(y2)--(z2)--(x2);
	
	\draw (x0)--(xy1)--(y0) (x1)--(xy1)--(y1);
	\draw (x1)--(xy2)--(y1) (x2)--(xy2)--(y2);
	
	\draw (y0)--(yz1)--(z0) (y1)--(yz1)--(z1);
	\draw (y1)--(yz2)--(z1) (y2)--(yz2)--(z2);
	
	\draw (z0)--(zx1)--(x0) (z1)--(zx1)--(x1);
	\draw (z1)--(zx2)--(x1) (z2)--(zx2)--(x2);
	
	\draw (x0)--(x1)--(x2)--(s);
	\draw (y0)--(y1)--(y2)--(s);
	\draw (z0)--(z1)--(z2)--(s);
	
	\node at (150:2.7*\x) {$\Gamma(3)$};
	
	\begin{scope}[shift={(8,0)}]
	\node[vertex, fill=black!60] (s) at (0,0) [label=270:{$s$}] {} ;
	
	\node[vertex, fill=black] (x0) at (90:3*\x) [label=90:{$x_0$}] {};
	\node[vertex, fill=black] (y0) at (90+120:3*\x) [label=90+120:{$y_0$}] {};
	\node[vertex, fill=black] (z0) at (90+240:3*\x) [label=90+240:{$z_0$}] {};
	
	\node[vertex, fill=black!40] (x1) at (90:2*\x) {};
	\node[vertex, fill=black!40] (y1) at (90+120:2*\x) {};
	\node[vertex, fill=black!40] (z1) at (90+240:2*\x) {};
	
	\draw[thickedge] (x0)--(y0)--(z0)--(x0);
	\draw[thickedge] (x1)--(y1)--(z1)--(x1);
	
	\draw (x1)--(s) (y1)--(s) (z1)--(s);
	\draw[thickedge, dotted] (x0)--(x1) (y0)--(y1) (z0)--(z1);
	
	\node at (150+120:0.25 + 2*\y) {$\delta$};
	
	\node at (150:2.7*\x) {$\Gamma(\delta)$};
	\end{scope}
	
	\end{tikzpicture}
	\caption{Left: the graph $\Gamma(3)$. 
		The vertices $x_0$, $y_0$ and $z_0$ are shaded black. The innermost vertex $s$ and the vertices $x_2$, $y_2$ and $z_2$ are all shaded grey. 
		Right: the way we will normally display the graph $\Gamma(\delta)$.} 
	\label{fig_gamma_1}
\end{figure}
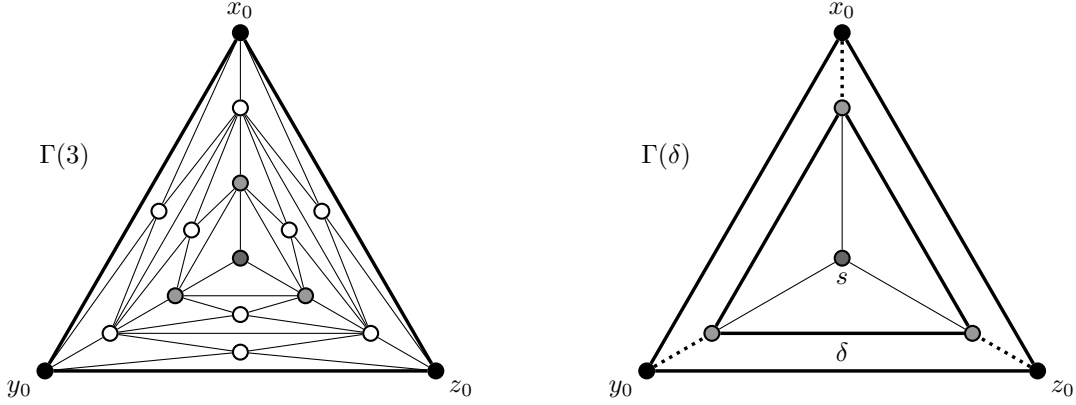

Let $\delta \geq 1$ be an integer, which we call the \textbf{depth} of the construction. 
Create a maximal plane graph $\Gamma(\delta)(f)$ (when the face $f$ in question is clear from context, we will just refer to the graph as $\Gamma(\delta)$) as follows:\\ 
Create the graph $P_\delta \times C_3$ and denote by $T_i$ the $i^{th}$ copy of $C_3$, with vertices $\{x_i, y_i, z_i\}$, where $i$ is in $[0,\delta-1]$. 
In the triangle $T_{\delta-1}$, place a vertex $s$, and make $s$ adjacent to the vertices $x_{\delta-1}$, $y_{\delta-1}$ and $z_{\delta-1}$. 
In the interior of each face bounded by a 4-cycle, add a single vertex, and make it adjacent to each vertex contained in the boundary of the face. 
(see Figure \ref{fig_gamma_1}).
For $\delta = 0$, let $\Gamma(\delta)$ be the triangle $T_0$ with vertex set $\{x_0, y_0, z_0\}$.
In order to build the desired maximal plane supergraph $G_f$ from $H$, we identify the vertices $x$, $y$ and $z$ of $H$ with the vertices $x_0$, $y_0$ and $z_0$ of $\Gamma(\delta)$.
We will call this operation of taking the union $H\cup \Gamma(\delta)$ and identifying $x$, $y$ and $z$ with $x_0$, $y_0$ and $z_0$ respectively the \textbf{glueing} of $\Gamma(\delta)$ along the face $f$.

\begin{obs}
	Consider the graph $\Gamma(\delta)$, where $\delta \geq 0$. 
	The vertex $x_0$ is distance at most $\delta$ from every vertex of $\Gamma(\delta)$, is distance exactly $\delta$ from $s$, $y_{\delta-1}$ and $z_{\delta-1}$, and is distance $\delta-1$ from $x_{\delta-1}$.
	\label{obs_dist_in_face}
\end{obs}

We will normally only be interested in what is happening in the faces $s,x_{\delta - 1}, y_{\delta - 1}$; $s,y_{\delta - 1}, z_{\delta - 1}$ and $s,x_{\delta - 1}, z_{\delta - 1}$ of $\Gamma(\delta)$, and will add vertices and edges inside these faces as needed. 
As such, we will leave out the additional clutter in our diagrams (see Figure \ref{fig_gamma_1}). 

For the following lemmas, let $H$ be a maximal plane graph, let $f:x,y,z$ be a face of $H$, and let $\alpha \geq \text{diam}(H)$ be an integer.
These lemmas will demonstrate how to construct the graph $G_f$ for some of the possible configurations of $\qcc_H(H[f])$. 
By Lemma \ref{lem:mpgiso}, the graph $H$ will always be an isometric subgraph of the graph $G_f$.
Therefore, if $u$ and $v$ are vertices of $H$, there is no ambiguity in referring to `the distance $d(u,v)$'.
%

\begin{lem}
	If there is some integer $k< \alpha$ such that every vertex of $\qcc_H(H[f])$ is distance $k$ from $H[f]$, then:
	Letting $\delta = \alpha-k$, and letting $G_f$ be the graph formed by gluing $\Gamma(\delta)$ to $H$ along $f$, every vertex of $\qcc_H(H[f])$ has eccentricity $\alpha$ in $G_f$, and no vertex of $H$ has eccentricity greater than $\alpha$ in $G_f$.
	\label{lem_const_1}
\end{lem}


\begin{proof}
	Per Observation \ref{obs_ecc_is_qcc}, the set $\ecc_H(H[f])$ is a subset of $\qcc_H(H[f])$, so every vertex of $H$ is within distance $k$ of $H[f]$. 
	By Observation \ref{obs_dist_in_face}, every vertex of $\Gamma(\delta)$ is within distance $\delta = \alpha - k$ of each vertex of $H[f]$, and so every vertex of $H$ is within distance $\alpha$ of every vertex of $\Gamma(\delta)$. 
	Consequently, every vertex $u$ in $H$ satisfies the inequality $e_{G_f}(u)\leq \alpha$. 
	
	Let $s$ be the unique vertex of $\Gamma(\delta)$ that is distance $\delta = \alpha - k$ from $H[f]$ and observe that $H[f]$ separates $s$ from $\qcc_H(H[f])$. 
	Since the vertices of $\qcc_H(H[f])$ are distance $k$ from $H[f]$, and all the vertices of $H[f]$ are distance $\delta = \alpha - k$ from $s$, every vertex of $\qcc_H(H[f])$ is distance $\alpha$ from $s$.
	Thus every vertex of $\qcc_H(H[f])$ has eccentricity $\alpha$ in $G_f$.
\end{proof}

In Lemma \ref{lem_const_1}, we assumed that $k$ was strictly less than $\alpha$. 
If $k$ is equal to $\alpha$, i.e., every vertex of $\qcc_H(H[f])$ is distance $\alpha$ from some vertex of $H[f]$, then every vertex of $\qcc_H(H[f])$ has eccentricity $\alpha$ in $H$. 
Therefore, if $k$ is equal to $\alpha$, choosing $G_f$ to be the graph $H$ will suffice.

We consider a small modification to the construction of $\Gamma(\delta)$. 
For a face $f:x,y,z$ of $H$ and positive integer $\delta$, construct $\Gamma(\delta)$ as usual. 
Denote by $\Gamma_x(\delta)$ the graph obtained as follows:\\
Place one additional vertex, call it $t_x$, in the face $f_x : s, y_{\delta-1}, z_{\delta-1}$ of $\Gamma(\delta)$. Then, make the vertex $t_x$ adjacent to all three vertices on the boundary of $f_x$. 
(Figure \ref{fig_gamma_x}) 

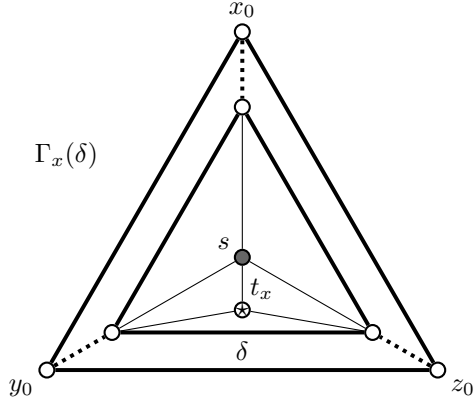
\begin{figure}[h]
	\centering
	\begin{tikzpicture}
	[inner sep=0.7mm, 
	vertex/.style={circle,thick,draw},
	dvertex/.style={rectangle,thick,draw, inner sep=1.3mm}, 
	thickedge/.style={line width=1.5pt}]
	
	\pgfmathsetmacro{\x}{1}
	\pgfmathsetmacro{\y}{0.5}
	
	\node[vertex, fill=black!60] (s) at (0,0) {} ;
	
	\node[vertex] (x0) at (90:3*\x) [label=90:{$x_0$}] {}; 
	\node[vertex] (y0) at (90+120:3*\x) [label=90+120:{$y_0$}] {};
	\node[vertex] (z0) at (90+240:3*\x) [label=90+240:{$z_0$}] {};
	
	\node[vertex] (x1) at (90:2*\x) {}; 
	\node[vertex] (y1) at (90+120:2*\x) {};
	\node[vertex] (z1) at (90+240:2*\x) {};
	
	\draw[thickedge] (x0)--(y0)--(z0)--(x0);
	\draw[thickedge] (x1)--(y1)--(z1)--(x1);
	
	\draw (x1)--(s) (y1)--(s) (z1)--(s);
	\draw[thickedge, dotted] (x0)--(x1) (y0)--(y1) (z0)--(z1);
	
	\node at (150+120:0.25 + 2*\y) {$\delta$};
	
	\node[vertex] (tx) at (270:0.7) {};
	\node at (270:0.7) {$\star$};
	\draw (tx)--(s) (tx)--(y1) (tx)--(z1);
	
	\node at (290:0.75) [label=:{$t_x$}] {};
	
	\node at (190:0.25) [label=:{$s$}] {};
	
	\node at (150:2.7*\x) {$\Gamma_{x}(\delta)$};
	
	\end{tikzpicture}
	\caption{The graph $\Gamma_x(\delta)$. The vertex $t_x$, 
		lies in the face $s,y_{\delta-1}, z_{\delta-1}$ of $\Gamma(\delta)$ and satisfies the equation $d(x_0, t_x) = \delta + 1$.} 
	\label{fig_gamma_x}
\end{figure}

\begin{obs}
	The vertex $x=x_0$ is distance at most $\delta$ from every vertex of $\Gamma_x(\delta)$ except for $t_x$, from which it is distance $\delta+1$. 
	Further, both $y$ and $z$ are distance exactly $\delta$ from $s$, and distance at most $\delta$ from every other vertex of $\Gamma_x(\delta)$.
	\label{obs_dist_in_gam_x}
\end{obs}

\begin{lem}
	Assume that $q(H[f]) = k$, and that both $(k+1)_y ^{x,z}$ and $(k+1)_z ^{x,y}$ are non-empty. 
	Let $\delta = \alpha-k-1$. 
	If $G_f$ is the graph formed by gluing $\Gamma_x(\delta)$ to $H$ along $f$, then every vertex of $\qcc_H(H[f])$ has eccentricity $\alpha$ in $G$, and no vertex of $H$ has eccentricity greater than $\alpha$ in $G_f$.
	\label{lem_const_2}
\end{lem}

\begin{proof}
	Since the set $(k+1)_y ^{x,z}$ is non-empty, some vertex of $H$ is distance $k+2$ from $x$. 
	Thus $\alpha \geq \text{diam}(H) \geq k+2$, so $\delta = \alpha - k - 1 \geq 1$. 
	By Corollary \ref{cor_twocnob_rest}, the set $(k+1)_x ^{y,z}$ is empty, and $(k)_x ^{y,z}$ is non-empty. 
	Further, by Lemmas \ref{lem_ak_rest} and \ref{lem_same_rest}, all the vertices in $\qcc_H(H[f]) - (k)_{x} ^{y,z}$ are distance $k+1$ from $H[f]$. 
	Since $\ecc_H(H[f])$ is a subset of $\qcc_H(H[f])$, all the vertices of $H-\qcc_H(H[f])$ are within distance $k$ of $H[f]$.
	
	From Observation \ref{obs_dist_in_gam_x}, and the facts listed in the previous paragraph, we can deduce that every vertex of $H$ has eccentricity at most $\alpha$ in $G_f$.
	Every vertex of $(k+1)$ is exactly distance $\alpha$ from $s$, and the vertices of $(k)_x ^{y,z}$ are distance $\alpha$ from $t_x$. 
	Thus every vertex of $\qcc_H(H[f])$ has eccentricity $\alpha$ in $G_f$.
\end{proof}

In much the same way that we constructed $\Gamma_x(\delta)$, we construct $\Gamma_{x,y}(\delta)$ from $\Gamma(\delta)$ by placing a vertex $t_x$ in the face $f_x : s,y_{\delta-1}, z_{\delta-1}$ of $\Gamma(\delta)$, and placing a vertex $t_y$ in the face $f_y : s, x_{\delta-1}, z_{\delta-1}$. Then, we add the three edges incident with each of $t_x$ and $t_y$ needed to ensure that the resulting graph is a maximal plane graph. 
See, for example, the left side of Figure \ref{fig_gamma_2}.

We can also construct $\Gamma_{x,y,z}(\delta)$ in the same way by adding a third vertex $t_z$ in the face $f_z : s, x_{\delta-1}, y_{\delta-1}$. 
See the right side of Figure \ref{fig_gamma_2} for an example of this construction. 
The vertices $t_x$, $t_y$ and $t_z$ are labelled with $\star$ symbols.

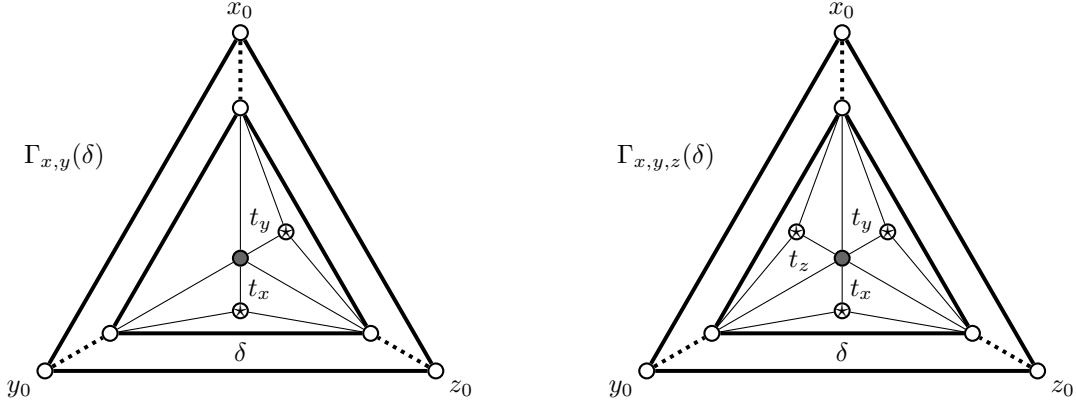
\begin{figure}[h]
	\centering
	\begin{tikzpicture}
	[inner sep=0.7mm, 
	vertex/.style={circle,thick,draw},
	dvertex/.style={rectangle,thick,draw, inner sep=1.3mm}, 
	thickedge/.style={line width=1.5pt}]
	
	\pgfmathsetmacro{\x}{1}
	\pgfmathsetmacro{\y}{0.5}
	
	\node[vertex, fill=black!60] (s) at (0,0) {} ;
	
	\node[vertex] (x0) at (90:3*\x) [label=90:{$x_0$}] {};
	\node[vertex] (y0) at (90+120:3*\x) [label=90+120:{$y_0$}] {};
	\node[vertex] (z0) at (90+240:3*\x) [label=90+240:{$z_0$}] {};
	
	\node[vertex] (x1) at (90:2*\x) {};
	\node[vertex] (y1) at (90+120:2*\x) {};
	\node[vertex] (z1) at (90+240:2*\x) {};
	
	\draw[thickedge] (x0)--(y0)--(z0)--(x0);
	\draw[thickedge] (x1)--(y1)--(z1)--(x1);
	
	\draw (x1)--(s) (y1)--(s) (z1)--(s);
	\draw[thickedge, dotted] (x0)--(x1) (y0)--(y1) (z0)--(z1);
	
	\node at (150+120:0.25 + 2*\y) {$\delta$};
	
	\node[vertex] (tx) at (270:0.7) {};
	\node at (270:0.7) {$\star$};
	\draw (tx)--(s) (tx)--(y1) (tx)--(z1);
	
	\node[vertex] (ty) at (30:0.7) {};
	\node at (30:0.7) {$\star$};
	\draw (ty)--(s) (ty)--(x1) (ty)--(z1);
	
	\node at (290:0.75) [label=:{$t_x$}] {};
	\node at (35:0.33) [label=:{$t_y$}] {};
	
	\node at (150:2.7*\x) {$\Gamma_{x,y}(\delta)$};
	
	\begin{scope}[shift={(8,0)}]
		\node[vertex, fill=black!60] (s) at (0,0) {} ;
		
		\node[vertex] (x0) at (90:3*\x) [label=90:{$x_0$}] {};
		\node[vertex] (y0) at (90+120:3*\x) [label=90+120:{$y_0$}] {};
		\node[vertex] (z0) at (90+240:3*\x) [label=90+240:{$z_0$}] {};
		
		\node[vertex] (x1) at (90:2*\x) {};
		\node[vertex] (y1) at (90+120:2*\x) {};
		\node[vertex] (z1) at (90+240:2*\x) {};
		
		\draw[thickedge] (x0)--(y0)--(z0)--(x0);
		\draw[thickedge] (x1)--(y1)--(z1)--(x1);
		
		\draw (x1)--(s) (y1)--(s) (z1)--(s);
		\draw[thickedge, dotted] (x0)--(x1) (y0)--(y1) (z0)--(z1);
		
		\node at (150+120:0.25 + 2*\y) {$\delta$};
		
		\node[vertex] (tx) at (270:0.7) {};
		\node at (270:0.7) {$\star$};
		\draw (tx)--(s) (tx)--(y1) (tx)--(z1);
		
		\node[vertex] (ty) at (30:0.7) {};
		\node at (30:0.7) {$\star$};
		\draw (ty)--(s) (ty)--(x1) (ty)--(z1);
		
		\node[vertex] (tz) at (150:0.7) {};
		\node at (150:0.7) {$\star$};
		\draw (tz)--(s) (tz)--(x1) (tz)--(y1);
		
		\node at (290:0.75) [label=:{$t_x$}] {};
		\node at (35:0.33) [label=:{$t_y$}] {};
		\node at (210:0.65) [label=:{$t_z$}] {};
		
		\node at (150:2.7*\x) {$\Gamma_{x,y,z}(\delta)$};
	\end{scope}
	
	\end{tikzpicture}
	\caption{Left: the graph $\Gamma_{x,y}(\delta)$. Right: the graph $\Gamma_{x,y,z}(\delta)$. The vertex $s$ is grey, and the vertices $t_x$, $t_y$ and $t_z$ contain star $\star$ symbols.} 
	\label{fig_gamma_2}
\end{figure}

In the spirit of Observations \ref{obs_dist_in_face} and \ref{obs_dist_in_gam_x}, we note that for any vertex $p$ in $\{x,y,z\}$, we have that $d(p,t_p) = \delta + 1$, but $p$ is within distance $\delta$ of every other vertex of $\Gamma_{x,y,z}(\delta)$.

We will need one more type of modification of $\Gamma(\delta)$. 
We want to modify $\Gamma(\delta)$ in such a way as to have some vertex $t_{xy}$ which is distance $\delta$ from $z$ and distance $\delta + 1$ from $x$ and $y$. 
To this end, we first construct $\Gamma_{x,y}(\delta)$. 
To the face $f_{xy} : s,z_{\delta-1}, t_x$, add a vertex $t_{xy}$, and let $t_{xy}$ be adjacent to $s$, $z_{\delta - 1}$ and $t_x$.
Call the resulting graph $\Gamma_{x,y}^{xy}(\delta)$. 
The superscript $xy$ indicates that there exists a vertex distance $\delta+1$ from both $x$ and $y$. 
In a similar fashion, we can construct $\Gamma_{x,y,z}^{xy}(\delta)$, $\Gamma_{x,y,z}^{xy,yz}(\delta)$ and $\Gamma_{x,y,z}^{xy,yz,zx}(\delta)$. 
See Figure \ref{fig_gamma_3} for some of these constructions. 
The vertices $t_{xy}$, $t_{yz}$ and $t_{xz}$ are marked with $\bullet$ symbols.

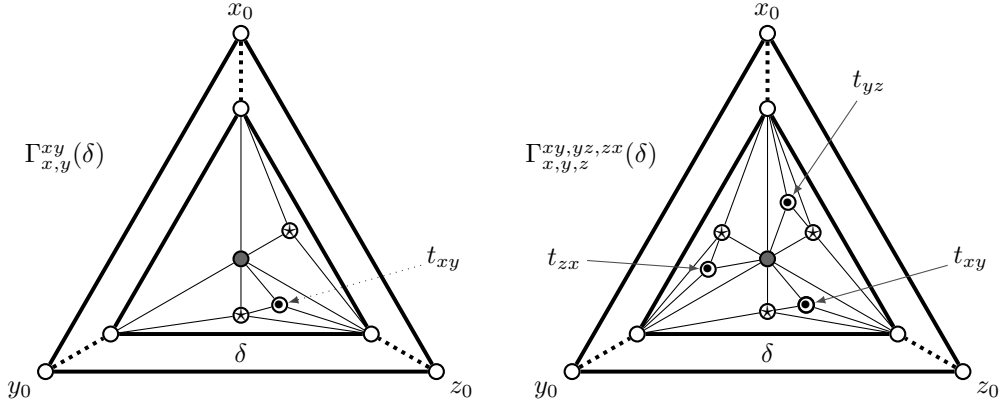
\begin{figure}[h]
	\centering
	\begin{tikzpicture}
	[inner sep=0.7mm, 
	vertex/.style={circle,thick,draw},
	dvertex/.style={rectangle,thick,draw, inner sep=1.3mm}, 
	thickedge/.style={line width=1.5pt}]
	
	\pgfmathsetmacro{\x}{1}
	\pgfmathsetmacro{\y}{0.5}
	
	\node[vertex, fill=black!60] (s) at (0,0) {} ;
	
	\node[vertex] (x0) at (90:3*\x) [label=90:{$x_0$}] {};
	\node[vertex] (y0) at (90+120:3*\x) [label=90+120:{$y_0$}] {};
	\node[vertex] (z0) at (90+240:3*\x) [label=90+240:{$z_0$}] {};
	
	\node[vertex] (x1) at (90:2*\x) {};
	\node[vertex] (y1) at (90+120:2*\x) {};
	\node[vertex] (z1) at (90+240:2*\x) {};
	
	\draw[thickedge] (x0)--(y0)--(z0)--(x0);
	\draw[thickedge] (x1)--(y1)--(z1)--(x1);
	
	\draw (x1)--(s) (y1)--(s) (z1)--(s);
	\draw[thickedge, dotted] (x0)--(x1) (y0)--(y1) (z0)--(z1);
	
	\node at (150+120:0.25 + 2*\y) {$\delta$};
	
	\node[vertex] (tx) at (270:0.75) {};
	\node at (270:0.75) {$\star$};
	\draw (tx)--(s) (tx)--(y1) (tx)--(z1);
	
	\node[vertex] (ty) at (30:0.75) {};
	\node at (30:0.75) {$\star$};
	\draw (ty)--(s) (ty)--(x1) (ty)--(z1);
	
	\node[vertex] (txy) at (310:0.8) {};
	\node at (310:0.8) {$_\bullet$};
	\draw (txy)--(s) (txy)--(z1) (txy)--(tx);
	
	\node at (150:2.7*\x) {$\Gamma_{x,y} ^{xy}(\delta)$};
	
	\node (xylab) at (0:2.7*\x) {$t_{xy}$};
	
	\draw[black!75, dotted, ->] (xylab)--(txy);
	
	
	\begin{scope}[shift={(7,0)}]
		\node[vertex, fill=black!60] (s) at (0,0) {} ;
		
		\node[vertex] (x0) at (90:3*\x) [label=90:{$x_0$}] {};
		\node[vertex] (y0) at (90+120:3*\x) [label=90+120:{$y_0$}] {};
		\node[vertex] (z0) at (90+240:3*\x) [label=90+240:{$z_0$}] {};
		
		\node[vertex] (x1) at (90:2*\x) {};
		\node[vertex] (y1) at (90+120:2*\x) {};
		\node[vertex] (z1) at (90+240:2*\x) {};
		
		\draw[thickedge] (x0)--(y0)--(z0)--(x0);
		\draw[thickedge] (x1)--(y1)--(z1)--(x1);
		
		\draw (x1)--(s) (y1)--(s) (z1)--(s);
		\draw[thickedge, dotted] (x0)--(x1) (y0)--(y1) (z0)--(z1);
		
		\node at (150+120:0.25 + 2*\y) {$\delta$};
		
		\node[vertex] (tx) at (270:0.7) {};
		\node at (270:0.7) {$\star$};
		\draw (tx)--(s) (tx)--(y1) (tx)--(z1);
		
		\node[vertex] (ty) at (30:0.7) {};
		\node at (30:0.7) {$\star$};
		\draw (ty)--(s) (ty)--(x1) (ty)--(z1);
		
		\node[vertex] (tz) at (150:0.7) {};
		\node at (150:0.7) {$\star$};
		\draw (tz)--(s) (tz)--(x1) (tz)--(y1);
		
		\node[vertex] (txy) at (310:0.8) {};
		\node[vertex, inner sep = 0.1] (txy) at (310:0.8) {$_\bullet$};
		\draw (txy)--(s) (txy)--(z1) (txy)--(tx);
		
		
		\node[vertex] (tyz) at (70:0.8) {};
		\node at (70:0.8) {$_\bullet$};
		\draw (tyz)--(s) (tyz)--(x1) (tyz)--(ty);
		
		
		\node[vertex] (tzx) at (190:0.8) {};
		\node at (190:0.8) {$_\bullet$};
		\draw (tzx)--(s) (tzx)--(y1) (tzx)--(tz);
		

		\node (xylab) at (0:2.7*\x) {$t_{xy}$};
		\node (yzlab) at (60:2.7*\x) {$t_{yz}$};
		\node (zxlab) at (180:2.7*\x) {$t_{zx}$};
		
		\draw[black!70, ->] (xylab)--(txy);
		\draw[black!70, ->] (yzlab)--(tyz);
		\draw[black!70, ->] (zxlab)--(tzx);

		\node at (150:2.7*\x) {$\Gamma_{x,y,z} ^{xy,yz,zx}(\delta)$};
	\end{scope}
	
	\end{tikzpicture}
	\caption{Left: the graph $\Gamma_{x,y}^{xy}(\delta)$. Right: the graph $\Gamma_{x,y,z}^{xy,yz,zx}(\delta)$. 
	} 
	\label{fig_gamma_3}
\end{figure}

\begin{obs}[Distances in $\Gamma$ graphs]
	Consider any graph $\Gamma(\delta)$ (with possible subscripts and superscripts) of depth $\delta \geq 1$ that we have constructed thus far, which has vertices $x_0$, $y_0$ and $z_0$ on the boundary of its outermost triangle. 
	The vertex $s$ inside the innermost triangle will be distance $\delta$ from each of $x_0$, $y_0$ and $z_0$. 
	Any vertex $t_p$, where $p$ is in $\{x,y,z\}$, will be distance $\delta+1$ from $p$ and distance $\delta$ from each of the other two vertices in $\{x,y,z\}$. 
	Any vertex $t_{pq}$, where $p$ and $q$ are distinct elements of $\{x,y,z\}$, will be distance $\delta + 1$ from both $p$ and $q$, and distance $\delta$ from the remaining vertex of $\{x,y,z\}$.
	\label{obs_dist_with_modif}
\end{obs}

In the proof of the main Theorem (Theorem \ref{thm_mpg_qef_crit}), we will construct, for each face $f$ of $H$, a plane supergraph $G_f$ of $H$ such that the quasi-eccentric vertices of $H[f]$ have eccentricity $\alpha$ in $G_f$. 
In order to construct a plane supergraph $G$ of $H$ in which every vertex of $H$ has eccentricity $\alpha$, we will take the union of all the graphs $G_f$, where $f$ is a face of $H$. 
Theorem \ref{thm_glueing_graph_lem} demonstrates that this union has the properties we desire.

\begin{thm}
	Let $\{G_m = (V_m, E_m) : m\in \{1,\dots, n\} \}$ be a finite collection of connected graphs such that for any $m$ and $l$ in $\{1,\dots,n \}$, the pairwise intersection $V_m \cap V_l$ is the same nonempty set $S$ (i.e., the set $S$ does not depend on the choice of integers $m$ and $l$). 
	Define the graphs $G=\bigcup\{G_m : m\in \{1,\dots, n\} \}$ and $H =\bigcap\{G_m : m\in \{1,\dots, n\} \}$ (Note that $V(H) = S$). 
	If both of the following conditions hold:
	\begin{itemize}
		\item[(1)] The graph $H$ is connected, and
		\item[(2)] $H$ is an isometric subgraph of $G_m$ for all $m$ in $\{1,\dots,n \}$,
	\end{itemize}
	then each $G_m$ is isometric in $G$, and every vertex $u$ in $H$ satisfies
	${e_G(u) = \max\{e_{G_m}(u) : m \in \{1,\dots,n\} \}}.$
	\label{thm_glueing_graph_lem}
\end{thm}
\begin{proof}
	We begin by proving that $G_1$ is an isometric subgraph of $G$. 
	Assume that $\{G_m : m\in \{1,\dots, n\} \}$ satisfies the two conditions of the theorem, and let $u$ and $v$ be vertices of $G_1$. 
	By assumption, the set $V(H) = S$ is non-empty.
	Since each graph $G_m$ is connected, and the intersection $H$ of the family of graphs is non-empty, the graph $G$ is connected. 
	Let $P$ be a $u-v$ geodesic in $G$. 
	In order to prove that $G_1$ is isometric in $G$, it suffices to prove that there exists a path $P'$ in $G_1$ such that $\ell(P')\leq \ell(P)$. 
	If $P$ itself is contained in $G_1$, then we set $P'=P$. 
	We thus assume without loss of generality that $P$ is not contained in $G_1$. 
	Observe that if $xy$ is an edge of $G$, then it is not possible that $x$ is contained in $G_m-H$ and $y$ is contained in $G_l-H$ for distinct integers $m$ and $l$. 
	Thus we can partition the edges of $P$ into a number of smaller paths $Q_0, Q_1, \dots , Q_i$ such that $P = Q_0\cup Q_1\cup \dots \cup Q_i$ satisfies the following:
	\begin{itemize}
		\item[(1)] the path $Q_0$ starts at $u$ and ends in $H$, and is contained in $G_1$,
		\item[(2)] the path $Q_i$ starts in $H$ and ends at $v$, and is contained in $G_1$,
		\item[(3)] for each integer $j$ in $\{1, 2, \dots, i-1 \}$, the path $Q_j$ starts and ends in $H$, and is contained entirely in some graph $G_{m_j}$.
	\end{itemize}
	Since for each $j$ in $\{1, \dots , i-1 \}$ the path $Q_j$ starts and ends in $H$, and $H$ is isometric in each $G_m$, there exists some path $R_j$ in $H$ with the same starting and ending vertices as $Q_j$ such that $\ell(R_j)\leq \ell(Q_j)$. 
	We can thus let $P' = Q_0\cup R_1\cup \dots R_{i-1}\cup Q_i$, completing the proof that $G_1$ is isometric in $G$. 
	In the same way, every graph $G_m$ is also isometric in $G$.
	
	We now prove that the eccentricity in $G$ of a vertex $u$ in $H$ is the maximum of the eccentricities $e_{G_1}(u)$, $e_{G_2}(u)$, \dots, $e_{G_n}(u)$. 
	Assume that $\{G_m : m\in \{1,\dots, n\} \}$ satisfies the conditions of the theorem, and let $u$ be a vertex of $H$.
	If $v$ is a eccentric vertex of $u$ in $G$, then since $V(G)=\bigcup\{V_m : m\in \{1,\dots, n\} \}$, the vertex $v$ belongs to $G_m$ for some integer $m$ in $\{1,\dots, n \}$. 
	Any $u-v$ path in $G_m$ is itself a $u-v$ path in $G$, so $d_G(u,v)\leq d_{G_m}(u,v)$. 
	Consequently,
	\[
	e_G(u) \leq \max\{ e_{G_m}(u) : m\in \{1,\dots, n\} \}.
	\]
	For each integer $m$ in $\{1,\dots,n \}$, let $v_m$ be an eccentric vertex of $u$ in $G_m$. 
	Because each $G_m$ is isometric in $G$, we have that $d_G(u,v_m)=d_{G_m}(u,v_m)$, and so
	\[
	e_G(u) \geq \max\{ e_{G_m}(u) : m\in \{1,\dots, n\} \},
	\] 
	completing the proof.
\end{proof}

\begin{thm}[The quasi-eccentric face criterion in maximal planar graphs]
	Let $H$ be a maximal plane graph of diameter $d$, and let $\alpha \geq d$ be an integer. 
	The graph $H$ is a subgraph of some maximal plane graph $G$, such that $V(H)$ is a subset of $\fe_G(\alpha)$, if and only if for all $u$ in $H$ that satisfy $e_H(u)< \alpha$, there exists a face $f$ of $H$ such that $u$ is in $\qcc_H(H[f])$.
	\label{thm_mpg_qef_crit}
\end{thm}

\begin{proof}
	The necessity of the condition follows from Theorem \ref{thm_qef_criterion}. 
	It only remains to prove sufficiency. 
	Assume that for all vertices $u$ of $H$ that have eccentricity less than $\alpha$ in $H$, there exists a face $f$ of $H$ such that $u$ is in $\qcc_H(H[f])$. 
	
	We will construct the graph $G$.
	For each face $f$ of $H$, we will create a plane supergraph $G_f$ of $H$ using one of the $\Gamma$ constructions described in this section, and glueing the $\Gamma$ graph to $H$ along the boundary of $f$. 
	The graph $G_f$ will be chosen such that each vertex of $\qcc_H(H[f])$ will have eccentricity exactly $\alpha$ in $G_f$, and each vertex of $H$ will have eccentricity at most $\alpha$ in $G_f$. 
	Further, given any two faces $f_1$ and $f_2$ of $H$, the intersection $G_{f_1}\cap G_{f_2}$ will be the graph $H$, which is isometric in any plane graph which contains it. 
	By Theorem \ref{thm_glueing_graph_lem}, setting $G=\bigcup\{G_f : f \text{ a face of } H \}$ will complete the proof.
	As such, all that we need to do now is to construct the graphs $G_f$, where $f$ is a face of $H$.
	
	Let $f:x,y,z$ be a face of $H$, and let $q(H[f]) = k$. 
	By Lemma \ref{lem_mpg_qef_diffofone}, every vertex of $\qcc_H(H[f])$ is either distance $k$ from $H[f]$, or distance $k+1$ from $H[f]$. 
	Up to relabelling of the vertices $x$, $y$ and $z$, the following list describes all possible configurations of $\qcc_H(H[f])$:
	
	\begin{itemize}[leftmargin=50pt]
		\item[Case 1:] Every vertex of $\qcc_H(H[f])$ is distance $k$ from $H[f]$. 
		That is to say, the sets $\qcc_H(H[f])$ and $(k)$ are equal.
		\item[Case 2:] The sets $(k+1)_x ^{y,z}$, $(k+1)_y ^{x,z}$ and $(k+1)_z ^{x,y}$ are all non-empty.
		\item[Case 3:] The sets $(k+1)_x ^{y,z}$ and $(k+1)_y ^{x,z}$ are non-empty, but $(k+1)_z ^{x,y}$ is empty.
		\item[Case 4.1:] The set $(k+1)_x ^{y,z}$ is non-empty, and the sets $(k+1)_y ^{x,z}$, $(k+1)_z ^{x,y}$ and $(k)^{x} _{y,z}$ are empty.
		\item[Case 4.2:] The sets $(k+1)_x ^{y,z}$ and $(k)_{y,z} ^x$ are non-empty, but the sets $(k+1)_y ^{x,z}$ and $(k+1)_z ^{x,y}$ are empty.
		\item[Case 5.1:] The sets $(k+1)_x ^{y,z}$, $(k+1)_y ^{x,z}$ and $(k+1)_z ^{x,y}$ are all empty.
		However, each of the sets $(k+1)_{x,y} ^z$, $(k+1)_{x,z} ^y$ and $(k+1)_{z,y} ^x$ is non-empty.
		\item[Case 5.2:] The sets $(k+1)_x ^{y,z}$, $(k+1)_y ^{x,z}$ and $(k+1)_z ^{x,y}$ are all empty. 
		Further, the sets $(k+1)_{x,y} ^z$ and $(k+1)_{x,z} ^y$ are non-empty, but the set $(k+1)_{z,y} ^x$ is empty.
		\item[Case 5.3:] The sets $(k+1)_x ^{y,z}$, $(k+1)_y ^{x,z}$ and $(k+1)_z ^{x,y}$ are all empty. 
		On the other hand, the set $(k+1)_{x,y} ^z$ is non-empty, but the sets $(k+1)_{z,y} ^x$ and $(k+1)_{x,z} ^y$ are empty.
		\item[Case 6:] The set $(k+1)_{x,y,z}$ is non-empty. 
		Further, every vertex of $\qcc_H(H[f])$ at distance $k+1$ from $H[f]$ is contained in the set $(k+1)_{x,y,z}$.
	\end{itemize}
	
	We show that for each configuration in the list presented above, we can construct the maximal plane graph $G_f$ by placing another maximal plane graph $\Gamma_f$ inside the face $f$, and glueing $\Gamma_f$ to $H$ along $f$.
	If $f_1$ and $f_2$ are two distinct faces of $H$, it is clear that $H$ will be a subgraph of both $G_{f_1}$ and $G_{f_2}$.
	Since $H$ is a maximal plane graph, it will be an isometric subgraph of both $G_{f_1}$ and $G_{f_2}$, per Lemma \ref{lem:mpgiso}.
	Because the graphs $\Gamma_{f_1}$ and $\Gamma_{f_2}$ lie in different faces of $H$, the intersection $V(G_{f_1})\cap V(G_{f_2})$ is exactly the set $V(H)$.
	
	\textit{Case 1:}\\
	If all vertices of $\qcc_H(H[f])$ are distance $k$ from $H[f]$, then define $\delta = \alpha - k$. If $\delta = 0$, then $k = \alpha$ and every vertex of $\qcc_H(H[f])$ is distance $k = \alpha$ from $H[f]$. 
	Thus every vertex of $\qcc_H(H[f])$ has eccentricity $\alpha$ in $H$, so choose $\Gamma_f = H$. If $\delta > 0$, then by Lemma \ref{lem_const_1}, we can choose $\Gamma_f = \Gamma(\delta)$.
	
	\textit{Case 2:}\\
	By Corollary \ref{cor_no_three_c_rest}, this case is not possible.
	
	\textit{Case 3:}\\
	Let $\delta = \alpha - k - 1$. 
	Since the set $(k+1)^{z}$ is non-empty, there exists a vertex in $H$ which is distance $k+2$ from $y$. 
	Thus $\alpha \geq \text{diam}(H) \geq k+2$, and so $\delta \geq 1$. 
	By Lemma \ref{lem_const_2}, We can choose $\Gamma_f = \Gamma_{z}(\delta)$. 
	
	\textit{Case 4.1:}\\
	Let $\delta = \alpha - k - 1$ and $\Gamma_f = \Gamma_{y,z}(\delta)$. 
	By the same reasoning used in Case 3, we have $\delta \geq 1$. 
	Vertices in the sets $(k)_y ^{x,z}$ and $(k)_z ^{x,y}$ are distance $\alpha$ from the vertices $t_y$ and $t_z$ respectively. 
	By Lemmas \ref{lem_ak_rest}, \ref{lem_same_rest} and \ref{lem_cforb_rest}, all the vertices in $\qcc_H(H[f]) - ((k)_y ^{x,z}\cup(k)_z ^{x,y})$ must be in the set $(k+1)$, and thus distance $\delta + k + 1 = \alpha$ from $s$. 
	Since the sets $(k+1)_y ^{x,z}$ and $(k+1)_z ^{x,y}$ are empty, every vertex of $\qcc_H(H[f])$ has eccentricity at most $\alpha$ in $G_f$, by Observation \ref{obs_dist_with_modif}. 
	As $\ecc(H[f])$ is a subset of $\qcc_H(H[f])$, we must have that $e(H[f]) = k+1$, so every vertex of $H-\qcc_H(H[f])$ is within distance $k$ of $H[f]$.
	Thus each vertex of $H-\qcc_H(H[f])$ is within distance $\alpha$ of every vertex of $G_f$. 
	As such, every vertex of $H$ has eccentricity at most $\alpha$ in $G_f$ by Observation \ref{obs_dist_with_modif}.
	
	\textit{Case 4.2:}\\
	Let $\delta = \alpha - k -1$ and $\Gamma_f = \Gamma_{y,z}^{yz}(\delta)$. 
	As in Case 3, we have that $\delta \geq 1$. 
	By Observation \ref{obs_dist_with_modif}, the vertices in the set $(k)_{y,z} ^x$ are distance $\alpha$ from the vertex $t_{xy}$, and distance at most $\alpha$ from every other vertex of $G_f$. 
	The rest of this case follows in the same way as Case 4.1.
	
	\textit{Case 5.1:}\\
	Let $\delta = \alpha - k - 1$ (hence $\delta \geq 1$) and $\Gamma_f = \Gamma_{x,y,z}(\delta)$. 
	By Lemmas \ref{lem_ak_rest} and \ref{lem_same_rest}, every vertex in the set $(k)$ is contained in $(k)_x ^{y,z} \cup (k)_y ^{x,z} \cup (k)_z ^{x,y}$. 
	By Lemma \ref{lem_mpg_qef_diffofone}, every vertex of $\qcc_H(H[f]) - (k)$ is in the set $(k+1)$.
	As a consequence of the previous two sentences, every vertex of $\qcc_H(H[f])$ is distance $\alpha$ from some vertex in the set $\{s, t_x, t_y, t_z \}$. 
	By Observation \ref{obs_dist_with_modif}, every vertex of $\qcc_H(H[f])$ is distance at most $\alpha$ from each of the vertices $t_x$, $t_y$ and $t_z$. 
	The argument that every other vertex of $H$ is distance at most $\alpha$ from any vertex of $\Gamma_f$ is the same as the argument used in Case 4.1.
	
	\textit{Case 5.2:}\\
	Let $\delta = \alpha - k - 1$ (and note again that we have $\delta \geq 1$) and $\Gamma_f = \Gamma_{x,y,z}^{yz}(\delta)$. 
	This case is the same as Case 5.1, apart from the fact that the vertices in the set $(k)_{y,z} ^x$ will be distance $\alpha$ from the vertex $t_{yz}$.
	
	\textit{Case 5.3:}\\
	Let $\delta = \alpha - k - 1$ (and observe that $\delta \geq 1$) and $\Gamma_f = \Gamma_{x,y,z}^{yz, zx}(\delta)$. 
	This case is similar to Case 5.2, noting that any vertices in the set $(k)_{x,z} ^y$ will be distance $\alpha$ from the vertex $t_{zx}$.
	
	\textit{Case 6:}\\
	Let $\delta = \alpha - k - 1$.
	If $\delta = 0$, then $\alpha = k+1$. 
	By Lemma \ref{lem_ak_rest}, every vertex of $\qcc_H(H[f])$ is distance $\alpha = k+1$ from some vertex of $H[f]$, so it suffices to let $G_f = H$.
	On the other hand, if $\delta \geq 1$, let $\Gamma_f = \Gamma_{x,y,z}^{xy, yz, zx}(\delta)$. 
	By Observation \ref{obs_dist_with_modif}, every vertex of $\qcc_H(H[f])$ is distance at least $\alpha$ from some vertex in the set $\{s, t_{xy}, t_{yz}, t_{xz} \}$. 
	Note that the only vertices distance $k+1$ from $H[f]$ are those in $(k+1)_{x,y,z}$, and every other vertex of $H$ is distance at most $k$ from $H[f]$.
	Thus every vertex of $\qcc_H(H[f])$ is within distance at most $\alpha$ of each vertex of $\Gamma_f$.
	
	To complete the construction, let $G = \bigcup \{G_f : f \text{ a face of } H \}$.
	It is clear that $G$ is a maximal plane supergraph of $H$. 
	In the case analysis presented above, we have shown that regardless of the configuration of $\qcc_H(H[f])$, the vertices of $\qcc_H(H[f])$ have eccentricity $\alpha$ in $G_f$, and that every vertex of $H$ has eccentricity at most $\alpha$ in $G_f$.
	Thus, by Theorem \ref{thm_glueing_graph_lem}, every vertex which is quasi-eccentric to some face of $H$ has eccentricity $\alpha$ in $G$. 
	By assumption, every vertex of $H$ is quasi-eccentric to some face of $H$. 
	Thus, every vertex of $H$ has eccentricity $\alpha$ in $G$.
\end{proof}

\begin{obs}
	Let $G=(V,E)$ be a connected graph, and $S$ a subset of $V$. Then the following inclusion holds:
	\[
	\{u\in V : u \text{ is an eccentric vertex of some vertex in } S \} \subseteq \qcc(S).
	\]
	\label{obs_all_ecc_are_qcc}
\end{obs}

\begin{cor}
	If $H$ is a maximal planar graph in which every vertex is an eccentric vertex, then for any integer $\alpha \geq \text{diam}(H)$, there exists a maximal planar graph $G$ into which $H$ embeds such that $H\subseteq \fe_G(\alpha)$.
	\label{cor_ecc_mpg_suff}
\end{cor}

\begin{proof}
	We prove that every vertex of $H$ is quasi-eccentric to some face of $H$. 
	Let $u$ be a vertex of $H$, and $v$ a vertex to which $u$ is eccentric. 
	There exists some face $f$ such that $v$ is on the boundary of $f$, so by Observation \ref{obs_all_ecc_are_qcc}, the vertex $u$ is quasi-eccentric to $H[f]$.
	The result now follows from Theorem \ref{thm_mpg_qef_crit}.
\end{proof}

\begin{cor}
	If $H$ is a maximal planar graph such that every vertex is eccentric to $H[f]$ for some face $f$ of $H$, then for any integer $\alpha \geq \text{diam}(H)$, there exists a maximal planar graph $G$ into which $H$ embeds such that $H\subseteq \fe_G(\alpha)$.
	\label{cor_face_mpg_suff}
\end{cor}

\begin{proof}
	By Observation \ref{obs_ecc_is_qcc}, every vertex is quasi-eccentric to some face (in particular, the face to which it is eccentric). 
	The result follows from Theorem \ref{thm_mpg_qef_crit}.
\end{proof}

\section{Refinements and corollaries of Theorem \ref{thm_mpg_qef_crit}}

Theorem \ref{thm_mpg_qef_crit} gives a necessary condition for a maximal planar graph $H$ to be a subgraph of the center of some (maximal) planar graph $G$ (as the center of a graph is always an equi-eccentric subgraph), and a sufficient condition for $H$ to be equi-eccentric, with eccentricity $\alpha$, in $G$. 
Theorem \ref{thm_equi_subgraph_center} shows that the construction used in Theorem \ref{thm_mpg_qef_crit} ensures that $\alpha$ is the lowest eccentricity present in $G$, and hence that Theorem \ref{thm_mpg_qef_crit} exactly characterises maximal planar graphs which are contained in the center of some (maximal) planar graph.
But first, we need a lemma.

\begin{lem}
	\label{lem_far_face}
	Let $H$ be a maximal plane graph of order at least 4 such that every vertex of $H$ is quasi-eccentric to some face of $H$. 
	Then every vertex $u$ of $H$ is quasi-eccentric to some face $f$ such that $u$ is not contained in the subgraph $H[f]$ induced by the boundary of $f$.
\end{lem}

\begin{proof}
	By the assumption, there exists a face to which the vertex $u$ is quasi-eccentric. 
	If this face does not contain $u$ in its boundary, we are done.
	Assume that the face to which $u$ is quasi-eccentric does contain $u$ in its boundary. 
	Call this face $f_1:u,y,z$. 
	Since $H$ is a maximal plane graph of order 4 or more, the edge $yz$ lies on two faces, the face $f_1: u, y, z$ and another face $f_2: x, y, z$.
	The vertex $u$ is quasi-eccentric to the vertex set $\{u, y, z\}$ of $H[f_1]$, and $d(u,x) > d(u,u) = 0$, so $u$ is also quasi-eccentric to the set $\{x,y,z\}$.
	Thus $u$ is quasi-eccentric to $H[f_2]$, completing the proof.
\end{proof}

\begin{thm}
	\label{thm_equi_subgraph_center}
	Let $H$ be a maximal plane graph of order at least 4 satisfying the condition of Theorem \ref{thm_mpg_qef_crit}, and let $\alpha \geq \text{diam}(H)$ be an integer.
	Denote by $G$ the graph constructed in Theorem \ref{thm_mpg_qef_crit} such that $V(H)\subseteq \mathcal{E}_G(\alpha)$.
	Then the radius of $G$ is $\alpha$, and so $H$ is contained in the center of $G$.
\end{thm}

\begin{proof}
	Let $H$ be a maximal plane graph satisfying the hypotheses of the theorem. 
	We will need the following claim. 
	
	\textit{Claim:} If $u$ is a vertex of $H$, there is a vertex $v$ of $G-H$ that is not contained in any face of $H$ that has $u$ in its boundary, and such that $d(u,v) = \alpha$.
	
	\textit{Proof of Claim:} By construction of $G$, any face $f$ of $H$ such that $u$ is in $\qcc_H(H[f])$ contains an eccentric vertex of $u$. 
	Thus it suffices to find a face $f:x,y,z$ such that $u$ is in $\qcc_H(H[f])$, but for which $u$ is not in $\{x,y,z\}$.
	By Lemma \ref{lem_far_face}, the desired face $f$ exists, proving the claim.
	
	We now prove the Theorem. 
	Since every vertex of $H$ has eccentricity $\alpha$ in $G$, it suffices to prove that any vertex of $G-H$ has eccentricity at least $\alpha$ in $G$.
	Let $w$ be a vertex of $G-H$, and let $f:x,y,z$ be the face of $H$ containing $w$.
	By the prior claim, there is some vertex $v$ which does not lie inside the face $f$, such that $d(x,v) = \alpha$.
	Since $H[f]$ separates $w$ and $v$, by Lemma \ref{lem_sep_different_components} any $w-v$ path $P$ has length:
	\[
	\ell(P) \geq d(w, H[f]) + d(v,H[f]) \geq 1 + (\alpha - 1) = \alpha
	\]
	Note that the distance $d(v,H[f]) \geq (\alpha - 1)$ since $d(x,v) = \alpha$ and $x$ is adjacent to every vertex of $H[f]$.
	Since $d(w,v) \geq \alpha$, we have $e(w)\geq \alpha$ in $G$, completing the proof.
\end{proof}

We can refine Theorems \ref{thm_mpg_qef_crit} and \ref{thm_equi_subgraph_center} further. 
Given a maximal planar graph $H$ which satisfies the hypotheses of Theorem \ref{thm_mpg_qef_crit}, we can find a planar (but not maximal planar) graph $PG$ such that $H$ is exactly the center of $PG$ (rather than just being contained in the center, as we obtain in Theorem \ref{thm_equi_subgraph_center}).

To do so, we need only make two small modifications to the proof of Theorem \ref{thm_mpg_qef_crit}, the latter of which was suggested by Dankelmann \cite{dankelmann:centerpc}.
First, we insist that the integer $\alpha$ satisfies $\alpha \geq \text{diam(H)} + 3$.
Note that this forces any occurrence of the integer $\delta$ in the proof of Theorem \ref{thm_mpg_qef_crit} to satisfy $\delta \geq 2$. 
Second, find all the vertices of $G-H$ which are adjacent to exactly two vertices of $H$. 
Remove these vertices from the maximal planar graph $G$ to obtain the planar graph $PG$
(See Figure \ref{fig_planar_mod}).

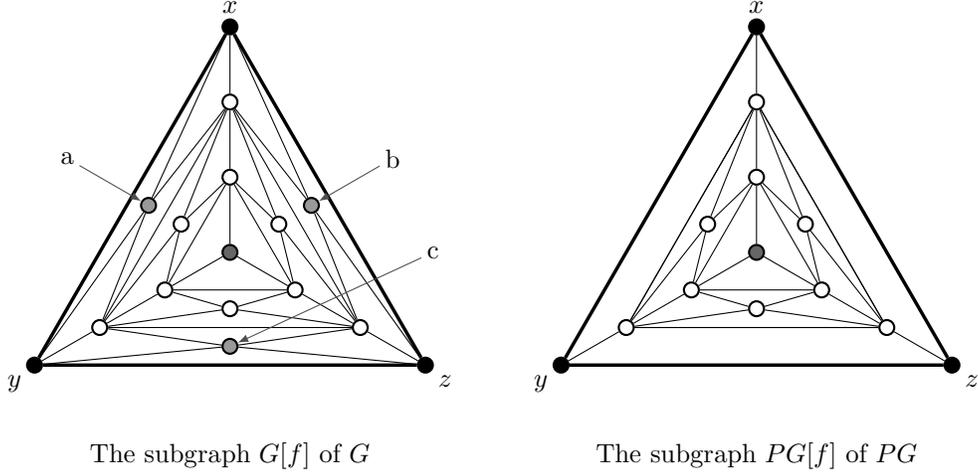
\begin{figure}[h]
	\centering
	\begin{tikzpicture}
		[inner sep=0.7mm, 
		vertex/.style={circle,thick,draw},
		dvertex/.style={rectangle,thick,draw, inner sep=1.0mm}, 
		thickedge/.style={line width=1.3pt}]
		
		\pgfmathsetmacro{\x}{1}
		\pgfmathsetmacro{\y}{0.5}
		
		\node[vertex, fill=black!60] (s) at (0,0) {};
		
		\node[vertex, fill=black] (x0) at (90:3*\x) [label=90:{$x$}] {};
		\node[vertex, fill=black] (y0) at (90+120:3*\x) [label=90+120:{$y$}] {};
		\node[vertex, fill=black] (z0) at (90+240:3*\x) [label=90+240:{$z$}] {};
		
		\node[vertex] (x1) at (90:2*\x) {};
		\node[vertex] (y1) at (90+120:2*\x) {};
		\node[vertex] (z1) at (90+240:2*\x) {};
		
		\node[vertex, fill=black!40] (xy1) at (150:0.25 + 2*\y) {};
		\node[vertex, fill=black!40] (yz1) at (150+120:0.25 + 2*\y) {};
		\node[vertex, fill=black!40] (zx1) at (150+240:0.25 + 2*\y) {};
		
		\node[vertex] (xy2) at (150:0.25 + 1*\y) {};
		\node[vertex] (yz2) at (150+120:0.25 + 1*\y) {};
		\node[vertex] (zx2) at (150+240:0.25 + 1*\y) {};
		
		\node[vertex] (x2) at (90:1*\x) {};
		\node[vertex] (y2) at (90+120:1*\x) {};
		\node[vertex] (z2) at (90+240:1*\x) {};
		
		\draw[thickedge] (x0)--(y0)--(z0)--(x0);
		\draw (x1)--(y1)--(z1)--(x1);
		\draw (x2)--(y2)--(z2)--(x2);
		
		\draw (x0)--(xy1)--(y0) (x1)--(xy1)--(y1);
		\draw (x1)--(xy2)--(y1) (x2)--(xy2)--(y2);
		
		\draw (y0)--(yz1)--(z0) (y1)--(yz1)--(z1);
		\draw (y1)--(yz2)--(z1) (y2)--(yz2)--(z2);
		
		\draw (z0)--(zx1)--(x0) (z1)--(zx1)--(x1);
		\draw (z1)--(zx2)--(x1) (z2)--(zx2)--(x2);
		
		\draw (x0)--(x1)--(x2)--(s);
		\draw (y0)--(y1)--(y2)--(s);
		\draw (z0)--(z1)--(z2)--(s);
		
		\node at (270:2.7*\x) {The subgraph $G[f]$ of $G$};
		
		\node (a) at (150:2.5*\x) {a};
		\node (b) at (30:2.5*\x) {b};
		\node (c) at (0:2.7*\x) {c};
		
		\draw[black!70, ->] (a)--(xy1);
		\draw[black!70, ->] (b)--(zx1);
		\draw[black!70, ->] (c)--(yz1);
		
		\begin{scope}[shift={(7,0)}]
			\node[vertex, fill=black!60] (s) at (0,0) {};
			
			\node[vertex, fill=black] (x0) at (90:3*\x) [label=90:{$x$}] {};
			\node[vertex, fill=black] (y0) at (90+120:3*\x) [label=90+120:{$y$}] {};
			\node[vertex, fill=black] (z0) at (90+240:3*\x) [label=90+240:{$z$}] {};
			
			\node[vertex] (x1) at (90:2*\x) {};
			\node[vertex] (y1) at (90+120:2*\x) {};
			\node[vertex] (z1) at (90+240:2*\x) {};
			
			
			\node[vertex] (xy2) at (150:0.25 + 1*\y) {};
			\node[vertex] (yz2) at (150+120:0.25 + 1*\y) {};
			\node[vertex] (zx2) at (150+240:0.25 + 1*\y) {};
			
			\node[vertex] (x2) at (90:1*\x) {};
			\node[vertex] (y2) at (90+120:1*\x) {};
			\node[vertex] (z2) at (90+240:1*\x) {};
			
			\draw[thickedge] (x0)--(y0)--(z0)--(x0);
			\draw (x1)--(y1)--(z1)--(x1);
			\draw (x2)--(y2)--(z2)--(x2);
			
			\draw (x0)--(y0) (x1)--(y1);
			\draw (x1)--(xy2)--(y1) (x2)--(xy2)--(y2);
			
			\draw (y0)--(z0) (y1)--(z1);
			\draw (y1)--(yz2)--(z1) (y2)--(yz2)--(z2);
			
			\draw (z0)--(x0) (z1)--(x1);
			\draw (z1)--(zx2)--(x1) (z2)--(zx2)--(x2);
			
			\draw (x0)--(x1)--(x2)--(s);
			\draw (y0)--(y1)--(y2)--(s);
			\draw (z0)--(z1)--(z2)--(s);
			
			\node at (270:2.7*\x) {The subgraph $PG[f]$ of $PG$};
		\end{scope}
		
	\end{tikzpicture}
	\caption{On the left is a possible subgraph $G[f]$ of $G$, obtained in the proof of Theorem \ref{thm_mpg_qef_crit}, where $f$ is a face of $H$. 
		The vertices $a$, $b$ and $c$ are each adjacent to exactly two vertices of $H$.
		On the right is $PG[f] = G[f]-\{a,b,c\}$, which is a subgraph of $PG$.}  
	\label{fig_planar_mod}
\end{figure}

\begin{thm}
	\label{thm_equi_planar_center}
	Let $H$ be a maximal plane graph of order at least 4 satisfying the hypotheses of Theorem \ref{thm_mpg_qef_crit}, and let $\alpha \geq \text{diam}(H) + 3$ be an integer.
	Let $G$ be the maximal planar graph constructed in Theorem \ref{thm_mpg_qef_crit}, and let $PG$ be the planar graph formed by removing from $G$ all vertices of $G-H$ which are adjacent to exactly two vertices of $H$.
	Then $H$ is the center of $PG$.
	\label{thm_cen_H_cen_PG}
\end{thm}

\begin{proof}
	By Lemma \ref{lem:mpgiso}, the graph $H$ is isometric in $PG$.
	
	\textit{Claim:} The graph $PG$ is an isometric subgraph of $G$.
	
	\textit{Proof of claim:}
	Note we obtained $PG$ from the maximal planar graph $G$ by only removing vertices of degree 4, none of which are adjacent in $G$ (by construction of $G$).
	Let $u_1$ and $u_k$ be vertices of $PG$ and let $P$ be a $u_1-u_k$ geodesic in $G$.
	Among all such geodesics, let $P$ contain the minimum number of vertices of $G-PG$.
	If $P$ contains no vertex of $G-PG$, we have proven the claim. 
	So assume to the contrary that $P$ has the form $P: u_1, \dots, a, b, c, \dots ,u_k$, where $b$ is a vertex removed in constructing $PG$. 
	Letting $b^*$ be either of the vertices in $N(b)-\{a,c\}$, we see that the path $P^*: u_1, \dots, a, b^*, c, \dots ,u_k$ is also a $u_1-u_k$ geodesic in $G$, and one that contains fewer vertices of $G-PG$, contradicting the minimality of $P$ and thus proving the claim.
	
	By the construction of $PG$, and the fact that $PG$ is isometric in $G$, every vertex of $H$ will have eccentricity $\alpha$ in $PG$ (by the same arguments as those outlined in the proof of Theorem \ref{thm_mpg_qef_crit}).
	Thus to prove the theorem, it suffices to prove that every vertex of $PG-H$ has eccentricity strictly greater than $\alpha$.
	
	First, we establish that each vertex $u$ of $H$ has an eccentric vertex which does not lie inside a face containing $u$ in its boundary. 
	This follows from Lemma \ref{lem_far_face} by the same argument as was used in Theorem \ref{thm_equi_subgraph_center} (note that since $\alpha \geq \text{diam}(H) + 3$, the desired eccentric vertex is not adjacent to any vertex of $H$, and hence is in $PG$).
	
	Now let $w$ be a vertex of $PG-H$, and denote by $f:x,y,z$ the face of $H$ containing $w$.
	There are two cases to consider.
	
	\textit{Case 1:} the vertex $w$ is adjacent to one of $x$, $y$ or $z$.
	Assume without loss of generality that $w$ is adjacent to $x$. 
	By the construction of $PG$, the vertex $w$ is not adjacent to $y$ or $z$.
	By the preceding paragraph, there is a vertex $v$ that is eccentric to $x$, such that $v$ does not lie in a face of $H$ containing $x$ on its boundary. 
	Because the set $V(H[f]) = \{x,y,z\}$ separates $w$ and $v$, and $w$ is not adjacent to $y$ or $z$, any $w-v$ path $P$ satisfies either:
	\[
	\ell(P) = d(w, x) + d(x,v) = 1 + \alpha > \alpha
	\]
	or, for some $t\in \{y,z\}$:
	\[
	\ell(P) = d(w, t) + d(t,v) \geq 2 + (\alpha - 1) > \alpha.
	\]
	
	Where $d(t,v) \geq (\alpha - 1)$, since $d(x,v) = \alpha$ and $x$ is adjacent to $t$.
	Since the distance $d(w,v) > \alpha$, the eccentricity $e(w) > \alpha$ in $PG$.
	
	\textit{Case 2:} the vertex $w$ is not adjacent to any of $x$, $y$ or $z$, and thus the distance $d(w, H[f]) \geq 2$.
	Similarly to Case 1, we have a vertex $v$ that does not lie in a face of $H$ that has $x$ on its boundary, and such that $d(x,v) = \alpha$.
	The set $H[f]$ separates $w$ and $v$, so by Lemma \ref{lem_sep_different_components}:
	\[
	d(w,v) \geq d(w, H[f]) + d(H[f],v) \geq 2 + (\alpha - 1) > \alpha.
	\]
	Thus $w$ has eccentricity strictly greater than $\alpha$, completing the proof.
\end{proof}

The following corollaries are immediate, but worth stating.

\begin{cor}
	A maximal plane graph $H$ is the center of some plane graph if and only if, for each vertex $u$ of $H$, there exists some face $f$ of $H$ such that $u$ is in $\qcc_H(H[f])$.
	\label{cor_cen_char}
\end{cor}

\begin{cor}
	A maximal plane graph $H$ is the center of some planar graph $G$ if and only if $H$ is an equi-eccentric subgraph of some maximal plane graph $G'$.
	\label{cor_centre_iff_equi}
\end{cor}

By Corollary \ref{cor_centre_iff_equi}, Corollaries \ref{cor_ecc_mpg_suff} and \ref{cor_face_mpg_suff} also give sufficient conditions for a maximal planar graph to be the center of some planar graph.

As Theorem \ref{thm_equi_in_center} below shows, the results of Theorems \ref{thm_equi_subgraph_center} and \ref{thm_equi_planar_center} should not be surprising, and do not depend strongly on the specific construction used in the proof of Theorem \ref{thm_mpg_qef_crit}. 

\begin{thm}
	\label{thm_equi_in_center}
	Let $H$ be a maximal plane graph which does not have a dominating face.
	If $H$ is an equi-eccentric subgraph of a plane graph $G$, then $H$ is a subgraph of the center of $G$.
\end{thm}

\begin{proof}
	Let $\alpha$ be the integer such that $V(H) \subseteq \mathcal{E}_G(\alpha)$.
	Assume to the contrary that there exists a vertex $u$ of $G-H$ such that $e(u) < \alpha$.
	Let $f:x,y,z$ be the face of $H$ containing $u$, and let $w$ be an eccentric vertex of $x$ in $G$.
	We consider two cases.
	
	\textit{Case 1:} The vertex $w$ lies inside the face $f$.
	Let $v$ be a vertex of $H$ which is not adjacent to any vertex of $H[f]$ (such a vertex exists by assumption).
	Then, since $H[f]$ separates $v$ and $w$, we have by Lemma \ref{lem_sep_different_components} that:
	\[
	d(v,w) \geq d(v, H[f]) + d(w, H[f]) \geq 2 + (\alpha - 1) > \alpha.
	\]
	This contradicts the assumption that all vertices of $H$ have the same eccentricity in $G$.
	
	\textit{Case 2:} The vertex $w$ lies outside the face $f$.
	Because $w$ is not in $f$, the subgraph $H[f]$ separates $u$ and $w$, so by Lemma \ref{lem_sep_different_components}:
	\[
	d(u,w) \geq d(u, H[f]) + d(w, H[f]) \geq 1 + (\alpha -1) \geq \alpha.
	\]
	This contradicts the assumption that $e(u) < \alpha$, completing the proof.
\end{proof}

\section{An application --- all small MPGs embed centrally}

In this section, we show that every maximal planar graph of order at most eight is the center of some planar graph.
We first need two lemmas, the first of which is very well known.

\begin{lem}
	If $G=(V,E)$ is a $k$-connected graph of diameter $d$, then $|V|\geq 2 + (d-1)k$.
	\label{lem_order_kappa_diam}
\end{lem}

\begin{proof}
	Let $u$ and $v$ be vertices of $G$ such that $d(u,v) = d$, and define $N_i = \{x\in V: d(u,x) = i\}$.
	It is clear that for any $i$ in $\{1,2, \dots, d-1 \}$, the set $N_i$ disconnects $G$, and so $|N_i| \geq k$.
	Thus $|V| \geq |\{u,v\}| + (d-1)|N_i| \geq 2 + (d-1)k$.
\end{proof}

\begin{lem}
	Let $G=(V,E)$ be a maximal planar graph of order five or more, with exactly one universal vertex $u$.
	If there exists a face of $G$ which dominates $V$ and does not contain $u$ in its boundary, then $G$ is a subgraph of the center of some maximal planar graph.
	\label{lem_domface}
\end{lem}

\begin{proof}
	Since $u$ is a universal vertex, the radius of $G$ is one, and every vertex of $G-u$ has eccentricity two.
	Let $f:x,y,z$ be the face of $G$ which dominates $V-\{u\}$. 
	Create a new graph $G'$ by adding a vertex $v$ to $f$, and making $v$ adjacent to each of $x$, $y$ and $z$.
	Since $f$ dominates $V-\{u\}$, every vertex of $V-\{u\}$ is distance at most two from $v$.
	Because $u$ and $v$ are not adjacent, but $u$ is adjacent to $x$, we have that $d(u,v) = 2$.
	The graph $G'$ is thus a self-centered maximal planar graph containing $G$ as a subgraph.
\end{proof}

We also need a quick observation.

\begin{obs}
	Let $G= (V,E)$ be a connected graph, let $S\subset V$ be a set which dominates $V$ and let $u$ be a vertex of $G$.
	If $u$ is not in $S$, then $u$ is quasi-eccentric to $S$.
	\label{obs_qcc_dom}
\end{obs}

\begin{proof}
	For any vertex $v$ in $V$, we have $d(v,S) \leq 1$, so there exists $t_v$ in $S$ such that $d(v, t_v) \leq 1$.
	Since $u$ is not in $S$, it satisfies $d(u,t_v) \geq 1$. 
	As $v$ was arbitrary, we conclude that $u$ is quasi-eccentric to $S$.
\end{proof}

\begin{thm}
	If a maximal planar graph $H$ has at most eight vertices, then there exists some maximal planar graph $G$ such that $H$ is a subgraph of the center of $G$.
	\label{thm_order_eight}
\end{thm}

\begin{proof}
	Let $H=(V,E)$ be a maximal planar graph such that $|V| \leq 8$. 
	Let $r$ and $d$ be the radius and diameter of $H$, respectively.
	All maximal planar graphs of order at most four are self-centered, so we may assume that $|V| \geq 5$.
	All maximal planar graphs of order at least five are three-connected, thus by Lemma \ref{lem_order_kappa_diam}, the diameter of $H$ satisfies $d \leq \frac{|V| + 1}{3} = 3$. 
	Since $K_n$ is non-planar for any $n\geq 5$, we may assume that the diameter of $H$ is at least two.
	If $H$ is self-centered, the Theorem holds trivially, so we assume $r<d$.
	
	\textit{Case 1: $H$ has radius 1 and diameter 2.}\\
	In this case, $H$ has at least one universal vertex.
	We break Case 1 into subcases, based on the number of universal vertices in $H$.
	
	\textit{Case 1.1:} $H$ has three (or more) universal vertices.\\
	Let $S$ be a set of three universal vertices. If $H-S$ has at least three vertices, then $H$ contains $K_{3,3}$ as subgraph, contradicting the planarity of $H$.
	If $H-S$ contains two adjacent vertices, then $H$ contains $K_5$ as a subgraph, contradicting planarity.
	Thus $H$ is $K_3 + \overline{K_2}$. 
	Figure \ref{fig_k3k2} shows that $K_3+\overline{K_2}$ is a subgraph of a self-centered maximal planar graph.
	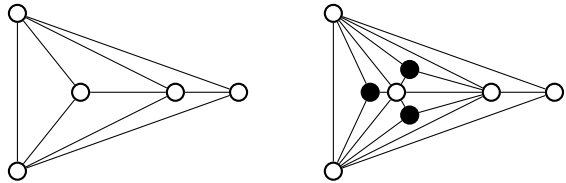
\begin{figure}[h]
\centering
	\begin{tikzpicture}
	[scale = 0.7, inner sep=0.8mm, 
	vertex/.style={circle,thick,draw}, 
	thickedge/.style={line width=1.5pt}] 
	
	\node[vertex](1) at (0,0) {};
	\node[vertex](2) at (1.8,0) {};
	\node[vertex](3) at (-1.2,1.5) {};
	\node[vertex](4) at (-1.2,-1.5) {};
	\node[vertex](5) at (3,0) {};
	
	\draw (2)--(3)--(4)--(2) (1)--(2) (1)--(3) (1)--(4) (5)--(2) (5)--(3) (5)--(4);
	
	\begin{scope}[shift={(6,0)}]
	
	\node[vertex](1) at (0,0) {};
	\node[vertex](2) at (1.8,0) {};
	\node[vertex](3) at (-1.2,1.5) {};
	\node[vertex](4) at (-1.2,-1.5) {};
	\node[vertex](5) at (3,0) {};
	
	\draw (2)--(3)--(4)--(2) (1)--(2) (1)--(3) (1)--(4) (5)--(2) (5)--(3) (5)--(4);
	
	\node[vertex, fill=black] (a) at (60:0.5) {};
	\node[vertex, fill=black] (b) at (180:0.5) {};
	\node[vertex, fill=black] (c) at (300:0.5) {};
	\end{scope}
	
	\draw (a)--(1) (a)--(2) (a)--(3);
	\draw (b)--(1) (b)--(3) (b)--(4);
	\draw (c)--(1) (c)--(2) (c)--(4);
	
	\end{tikzpicture}
	\caption{On the left is the graph $H = K_3 + \overline{K_2}$. On the right is a self-centered maximal planar graph $G$ containing $H$ as a subgraph.}
	\label{fig_k3k2}
\end{figure}
	
	\textit{Case 1.2:} $H$ has exactly two universal vertices.\\
	Let $u$ and $v$ be the two universal vertices of $H$.
	Since $H$ is a maximal planar graph of order at least five, the set $V-\{u,v\}$ contains a pair of adjacent vertices $x$ and $y$.
	The triangles $u,x,y$ and $v,x,y$ bound faces of $H$.
	Create a self-centered maximal planar graph $G$ containing $H$ as a subgraph as follows: 
	Add a vertex $s$ to the face bounded by $u,x,y$, and make $s$ adjacent to $u$, $x$ and $y$. 
	Similarly, add a vertex $t$ to the face bounded by $v,x,y$ and make $t$ adjacent to $v$, $x$ and $y$.
	Every vertex of $G-\{s,t\}$ is adjacent to $u$ and $v$, and hence every vertex of $G$ has eccentricity at most two.
	The only vertices ($x$ and $y$) adjacent to both $s$ and $t$ are not universal, so every vertex of $G$ has eccentricity at least two.
	Thus $G$ is self-centered. 
	
	\textit{Case 1.3:} $H$ has exactly one universal vertex.\\
	Let $v$ be the single universal vertex of $H$. 
	Since $H$ is a maximal planar graph, there is a cycle $C: u_1, u_2, \dots, u_k, u_1$ with $V(C) = N(v)$, and $4 \leq k \leq 7$. 
	Embed $H$ in the plane such that $v$ lies in the exterior region of $C$, and any chords of $C$ lie in the interior region. 
	Note that any face in the interior of $C$ does not contain $v$ in its boundary.
	In light of Lemma \ref{lem_domface}, it suffices to show that some face in the interior of $C$ dominates $V(C)$.
	
	If $k$ is either four or five, and $xy$ is a chord of $C$, then every vertex of $C$ is adjacent to $x$ or $y$. 
	Thus $V(C)$ is dominated by either of the faces incident with $xy$.
	As such, we may assume that $k$ is either six or seven.
	
	If $k=6$, and $u_iu_{i+3}$ is a 3-chord of $C$, then either of the faces incident with $u_iu_{i+3}$ dominate $C$. 
	So we assume that $C$ has no 3-chord. 
	Up to symmetry, the only way to triangulate the interior of $C$ without any 3-chords, is to create three 2-chords $u_1u_3$, $u_3u_5$ and $u_5u_1$.
	But then the triangular face $u_1, u_3, u_5$ dominates $V(C)$, so we assume that $k=7$ (see Figure \ref{fig_6cycle_chord}).
	
	\begin{figure}[h]
\centering
	\begin{tikzpicture}
	[scale = 0.7, inner sep=0.8mm, 
	vertex/.style={circle,thick,draw}, 
	thickedge/.style={line width=1.5pt}] 
	
	\node[vertex, fill = black] (0) at (0:2) [label=0:{$u_{i+3}$}] {};
	\node[vertex] (1) at (60:2) {};
	\node[vertex] (2) at (120:2) {};
	\node[vertex, fill = black] (3) at (180:2) [label=180:{$u_i$}] {};
	\node[vertex] (4) at (240:2) {};
	\node[vertex] (5) at (300:2) {};
	
	\draw[thickedge] (0)--(3);
	\draw (0)--(1)--(2)--(3)--(4)--(5)--(0);
	
	\begin{scope}[shift={(7,0)}]
	\node[vertex, fill = black] (0) at (0:2) [label=0:{$u_1$}] {};
	\node[vertex] (1) at (60:2) {};
	\node[vertex, fill = black] (2) at (120:2) [label=120:{$u_3$}] {};
	\node[vertex] (3) at (180:2) {};
	\node[vertex, fill = black] (4) at (240:2) [label=240:{$u_5$}] {};
	\node[vertex] (5) at (300:2) {};
	
	\draw[thickedge] (0)--(2)--(4)--(0);
	\draw (0)--(1)--(2)--(3)--(4)--(5)--(0);
	\end{scope}
	
	\end{tikzpicture}
	\caption{On the left is the cycle $C$ with a 3-chord, $u_iu_{i+3}$. Note that $\{u_i, u_{i+3}\}$ dominates $V(C)$.
	On the right is the cycle $C$ with three 2-chords. The set $\{u_1, u_3, u_5\}$ dominates $V(C)$.}
	\label{fig_6cycle_chord}
\end{figure}
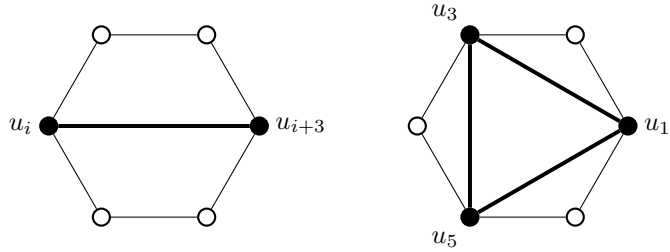
	
	We claim that the interior of $C: u_1, u_2, \dots, u_7, u_1$ cannot be triangulated by only 2-chords.
	Assume to the contrary that the interior of $C$ is triangulated, and that every edge in the interior of $C$ is a 2-chord of $C$.
	Up to a cyclic re-labeling of the vertices of $C$, we may assume $u_1u_3$ is a chord of $C$. 
	This chord lies on the boundary of two triangular faces of $H$, the face $u_1,u_2,u_3$, and another face $u_1, u_i, u_3$, for some $i$ in $\{4,5,6,7\}$. 
	If $i \leq 5$, then $u_1u_i$ is not a 2-chord. 
	If $i \geq 6$, then $u_3u_i$ is not a 2-chord. 
	In any case, we obtain a contradiction, proving the claim. 
	
	Thus if $k=7$, then $H$ must contain some 3-chord of $C$. 
	Up to re-labeling of the vertices of $C$, we may call this chord $u_1u_4$. 
	The chord $u_1u_4$ lies on the boundary of two triangular faces of $H$, one of which contains a vertex $u_j$, for $j$ in $\{5,6,7\}$. 
	For any choice of $j$ in $\{5,6,7\}$, the vertex $u_j$ either is, or is adjacent to $u_6$, and so the face $u_1,u_j,u_4$ dominates $V(C)$, completing the proof of Case 1.3.
	
	\textit{Case 2: $H$ has radius two and diameter three.}\\
	Let $u$ and $v$ be vertices of $H$ such that $d(u,v) = 3$, and let $N_i =\{x\in V : d(u,x) = i \}$.
	Per Lemma \ref{lem_order_kappa_diam} and the fact that $H$ is 3-connected, the order of $H$ is exactly eight.
	Further, since $N_1$ and $N_2$ are both $u-v$ separators, we have that $|N_1| \geq 3$ and $|N_2| \geq 3$.
	From these facts, it follows that $|N_1| = 3$, $|N_2| = 3$ and $N_3 = \{v\}$.
	We thus label vertices $N_1 = \{x_1, x_2, x_3\}$ and $N_2 = \{y_1, y_2, y_3\}$.
	It is clear that $N(u) = N_1$, but also note that $N(v) = N_2$.
	Certainly every neighbor of $v$ must belong to $N_2$, and were some vertex of $N_2$ not adjacent to $v$, then $N_2\cap N(v)$ would be a $u-v$ separator with fewer than three vertices, contradicting the 3-connectivity of $H$.
	
	Since $H$ is a maximal planar graph, both $N(u)$ and $N(v)$ must induce cycles in $H$.
	We claim that up to some relabeling of vertices, the edges $x_1y_1$, $x_2y_2$ and $x_3y_3$ all belong to $H$.
	It suffices to find a perfect matching of the bipartite subgraph $B$ that has vertex set $N(u) \cup N(v)$ and edge set $\{xy \in E : x\in N(u), y\in N(v) \}$.
	
	By Hall's Marriage Theorem \cite{hall:hall_marriage}, we need only show that for any set $S\subset N(v)$, we have $|S| \leq |N(S) \cap N(u)|$. 
	Because $N(v) = N_2$, every vertex of $N(v)$ is adjacent to some vertex of $N_1 = N(u)$, so $|\{y_i\}| \leq |N(y_i) \cap N(u)|$ for any singleton $\{y_i\}$ of $N(v)$.
	We claim it is not possible that both $|S| = 2$, and $|N(S) \cap N(u)| = 1$.
	Assume to the contrary it is, and let $x_k$ be the single vertex of $N(S) \cap N(u)$. 
	Let $y_j$ be the single vertex of $N(v)-S$, and observe that $\{x_k, y_j\}$ is a $u-v$ separator with only two vertices, contradicting the 3-connectivity of $H$ and proving the claim.
	It remains to show that if $|S| = 3$, then $|N(u) \cap N(S)| = 3$.
	Assume to the contrary that $S = N(v)$, but that the set $N(u) \cap N(S)$ has fewer than three vertices. 
	Then $N(u) \cap N(S)$ is a $u-v$ separator with fewer than three vertices, contradicting the 3-connectivity of $H$.
	Since $|S| \leq |N(S) \cap N(u)|$ for any $S\subseteq N(v)$, the subgraph $B$ of $H$ has a perfect matching. 
	Thus, up to some relabeling of the vertices, $H$ contains the edges $x_1y_1$, $x_2y_2$ and $x_3y_3$.
	
	We deduce that $H$ contains a subgraph $K$ which consists of two disjoint copies of $K_4$, joined by three disjoint edges (See Figure \ref{fig_k_construction}). 
	
	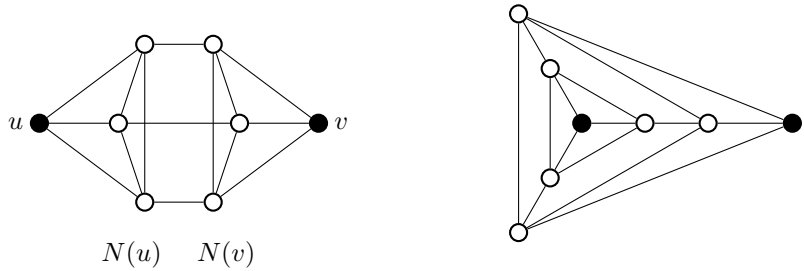
\begin{figure}[h]
\centering
	\begin{tikzpicture}
	[scale = 0.7, inner sep=0.8mm, 
	vertex/.style={circle,thick,draw}, 
	thickedge/.style={line width=1.5pt}] 
	
	\node[vertex, fill=black] (u) at (-0.3,0) [label=180:{$u$}] {};
	\node[vertex] (x1) at (1.7,1.5) {};
	\node[vertex] (x2) at (1.2,0) {};
	\node[vertex] (x3) at (1.7,-1.5) {};
	
	\node[vertex] (y1) at (3,1.5) {};
	\node[vertex] (y2) at (3.5,0) {};
	\node[vertex] (y3) at (3,-1.5) {};
	\node[vertex, fill=black] (v) at (5,0) [label=0:{$v$}] {};
	
	\node at (1.45, -2.5) {$N(u)$};
	\node at (3.25, -2.5) {$N(v)$};
	
	\draw (u)--(x1) (u)--(x2) (u)--(x3) (v)--(y1) (v)--(y2) (v)--(y3);
	\draw (x1)--(x2)--(x3)--(x1) (y1)--(y2)--(y3)--(y1);
	\draw (x1)--(y1) (x2)--(y2) (x3)--(y3);
	
	\begin{scope}[shift={(10,0)}]
		\node[vertex, fill=black] (u) at (0:0) {};
		\node[vertex] (x1) at (0:1.2) {};
		\node[vertex] (x2) at (120:1.2) {};
		\node[vertex] (x3) at (240:1.2) {};
		
		\node[vertex] (y1) at (0:2.4) {};
		\node[vertex] (y2) at (120:2.4) {};
		\node[vertex] (y3) at (240:2.4) {};
		
		\node[vertex, fill=black] (v) at (0:4) {};
		
		\draw (u)--(x1) (u)--(x2) (u)--(x3) (v)--(y1) (v)--(y2) (v)--(y3);
		\draw (x1)--(x2)--(x3)--(x1) (y1)--(y2)--(y3)--(y1);
		\draw (x1)--(y1) (x2)--(y2) (x3)--(y3);
	\end{scope}
	
	\end{tikzpicture}
	\caption{On the left is the spanning subgraph $K$ of $H$. On the right is a plane drawing of $K$ with no edges crossing. Since $K$ is 3-connected, any embedding of $K$ in the plane induces faces with the same boundaries.}
	\label{fig_k_construction}
\end{figure}
	
	It is easy to see that $K$ is 3-connected, and thus any two embeddings of $K$ as a plane graph induce the same set of faces boundaries \cite{whitney:congconn}.
	Since $H$ is planar, any edge of $H-K$ must be between two edges on the same face of $K$.
	Thus $H-K$ contains exactly three edges, one from each of the following three pairs: $\{x_1y_2, x_2y_1\}$, $\{x_2y_3, x_3y_2\}$ and $\{x_3y_1, x_1y_3\}$.
	Up to relabeling of the vertices, $H$ is one of just two possible graphs $H_1$ and $H_2$.
	The graph $H_1$ contains edges $x_1y_2$, $x_2y_3$ and $x_3y_1$, while $H_2$ contains $x_1y_2$, $x_3y_2$ and $x_3y_1$ (See Figure \ref{fig_h1_and_h2}).
	
	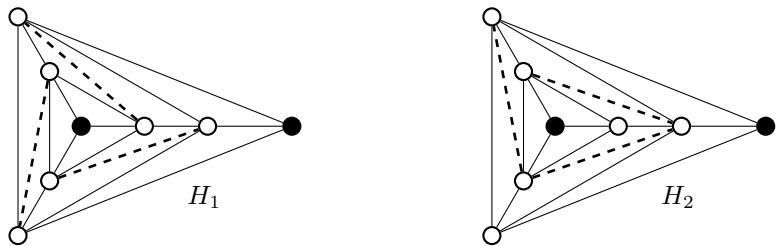
\begin{figure}[h]
\centering
	\begin{tikzpicture}
	[scale = 0.7, inner sep=0.8mm, 
	vertex/.style={circle,thick,draw}, 
	thickedge/.style={line width=1pt}] 
	
	\node[vertex, fill=black] (u) at (0:0) {};
	\node[vertex] (x1) at (0:1.2) {};
	\node[vertex] (x2) at (120:1.2) {};
	\node[vertex] (x3) at (240:1.2) {};
	
	\node[vertex] (y1) at (0:2.4) {};
	\node[vertex] (y2) at (120:2.4) {};
	\node[vertex] (y3) at (240:2.4) {};
	
	\node[vertex, fill=black] (v) at (0:4) {};
	
	\draw (u)--(x1) (u)--(x2) (u)--(x3) (v)--(y1) (v)--(y2) (v)--(y3);
	\draw (x1)--(x2)--(x3)--(x1) (y1)--(y2)--(y3)--(y1);
	\draw (x1)--(y1) (x2)--(y2) (x3)--(y3);
	
	\draw[thickedge, dashed] (x1)--(y2) (x2)--(y3) (x3)--(y1);
	\node at (-30:2.7) {$H_1$};
	
	\begin{scope}[shift={(9,0)}]
		\node[vertex, fill=black] (u) at (0:0) {};
		\node[vertex] (x1) at (0:1.2) {};
		\node[vertex] (x2) at (120:1.2) {};
		\node[vertex] (x3) at (240:1.2) {};
		
		\node[vertex] (y1) at (0:2.4) {};
		\node[vertex] (y2) at (120:2.4) {};
		\node[vertex] (y3) at (240:2.4) {};
		
		\node[vertex, fill=black] (v) at (0:4) {};
		
		\draw (u)--(x1) (u)--(x2) (u)--(x3) (v)--(y1) (v)--(y2) (v)--(y3);
		\draw (x1)--(x2)--(x3)--(x1) (y1)--(y2)--(y3)--(y1);
		\draw (x1)--(y1) (x2)--(y2) (x3)--(y3);
		
		\draw[thickedge, dashed] (x3)--(y2) (x2)--(y1) (x3)--(y1);
		\node at (-30:2.7) {$H_2$};
	\end{scope}
	
	\end{tikzpicture}
	\caption{On the left is the graph $H_1$, and on the right is $H_2$. The edges which do not belong to $K$ are indicated by dashed lines.}
	\label{fig_h1_and_h2}
\end{figure}
	
	We use Theorem \ref{thm_mpg_qef_crit} to complete the proof.
	We need to show that every vertex of $H_1$, other than $u$ and $v$, is quasi-eccentric to some face of $H_1$, and likewise for $H_2$. 
	
	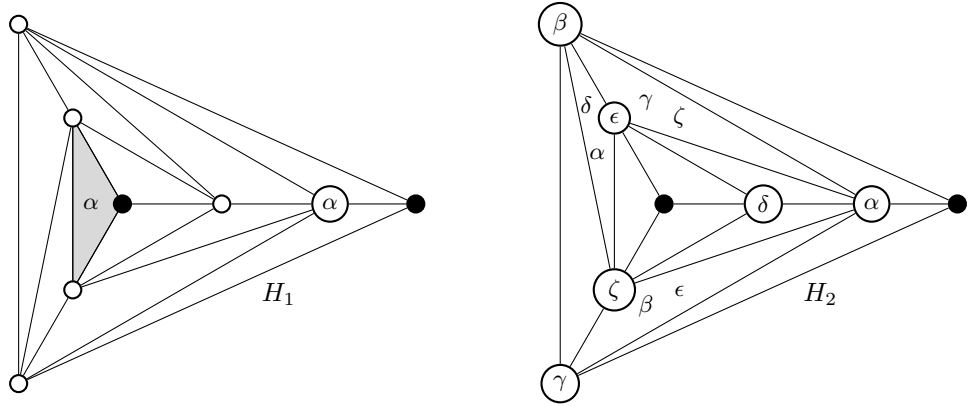
\begin{figure}[h]
\centering
	\begin{tikzpicture}
	[scale = 0.6, inner sep=0.8mm, 
	vertex/.style={circle,thick,draw}, 
	thickedge/.style={line width=1pt}] 
	
	\draw[fill=black!15] (0:0)--(120:2.2)--(240:2.2)--(0:0);
	
	\node[vertex, fill=black] (u) at (0:0) {};
	\node[vertex] (x1) at (0:2.2) {};
	\node[vertex, fill=white] (x2) at (120:2.2) {};
	\node[vertex, fill=white] (x3) at (240:2.2) {};
	
	\node[vertex] (y1) at (0:4.6) {\small $\alpha$};
	\node[vertex] (y2) at (120:4.6) {};
	\node[vertex] (y3) at (240:4.6) {};
	
	\node[vertex, fill=black] (v) at (0:6.5) {};
	
	\draw (u)--(x1) (u)--(x2) (u)--(x3) (v)--(y1) (v)--(y2) (v)--(y3);
	\draw (x1)--(x2)--(x3)--(x1) (y1)--(y2)--(y3)--(y1);
	\draw (x1)--(y1) (x2)--(y2) (x3)--(y3);
	
	\draw (x1)--(y2) (x2)--(y3) (x3)--(y1);
	
	\node at (-0.7, 0) {\small $\alpha$};
	
	\node at (-30:4.0) {$H_1$};
	
	\begin{scope}[shift={(12,0)}]
		\node[vertex, fill=black] (u) at (0:0) {};
		\node[vertex] (x1) at (0:2.2) {\small $\delta$};
		\node[vertex, fill=white] (x2) at (120:2.2) {\small $\epsilon$};
		\node[vertex, fill=white] (x3) at (240:2.2) {\small $\zeta$};
		
		\node[vertex] (y1) at (0:4.6) {\small $\alpha$};
		\node[vertex] (y2) at (120:4.6) {\small $\beta$};
		\node[vertex] (y3) at (240:4.6) {\small $\gamma$};
		
		\node[vertex, fill=black] (v) at (0:6.5) {};
		
		\draw (u)--(x1) (u)--(x2) (u)--(x3) (v)--(y1) (v)--(y2) (v)--(y3);
		\draw (x1)--(x2)--(x3)--(x1) (y1)--(y2)--(y3)--(y1);
		\draw (x1)--(y1) (x2)--(y2) (x3)--(y3);
		
		\draw (x3)--(y2) (x3)--(y1) (x2)--(y1);
		
		\node at (143:1.85) {\small $\alpha$};
		\node at (128:2.8) {\small $\delta$};
		
		\node at (100:2.3) {\small $\gamma$};
		\node at (80:1.95) {\small $\zeta$};
		
		\node at (260:2.3) {\small $\beta$};
		\node at (280:1.95) {\small $\epsilon$};
		
		\node at (-30:4.0) {$H_2$};
	\end{scope}
	
	\end{tikzpicture}
	\caption{In the graph $H_1$, the vertex marked $\alpha$ is quasi-eccentric to the darkened face marked $\alpha$. In the graph $H_2$, a face and a vertex share a label if the vertex is quasi-eccentric to the face.}
	\label{fig_h1_faces}
\end{figure}
	
	Observe Figure \ref{fig_h1_faces}.
	In $H_1$, the vertex labeled $\alpha$ is quasi-eccentric to the face labeled $\alpha$. 
	By symmetry, we thus see that every vertex of $H_1$ which has eccentricity at most two is quasi-eccentric to some face of $H_1$. 
	In $H_2$, a vertex is quasi-eccentric to a face if both the vertex and the face have the same label (e.g., the vertex $\delta$ is quasi-eccentric to the face with labels $\delta$ and $\alpha$).
	This is easy to check, as each face with a label in it dominates all of $V(H)$, and if a face and a vertex share a label, they are distance one apart, so quasi-eccentricity follows by Observation \ref{obs_qcc_dom}.
\end{proof}

We will show that the bound in Theorem \ref{thm_order_eight} is sharp, by exhibiting a maximal planar graph on nine vertices which fails the condition of Theorem \ref{thm_mpg_qef_crit}.

\begin{thm}
	There exists a maximal planar graph of order nine which is not a subgraph of the center of any maximal planar graph.
	\label{thm_mpg_order_nine}
\end{thm}

\begin{proof}
	Consider the graph $G_9$ in Figure \ref{fig_mpg_order_nine}.
	The diameter of $G_9$ is two, and $e(u) = 1$. 
	Thus, by Theorem \ref{thm_mpg_qef_crit}, it suffices to show that $u$ is not quasi-eccentric to any face of $G_9$.
	
	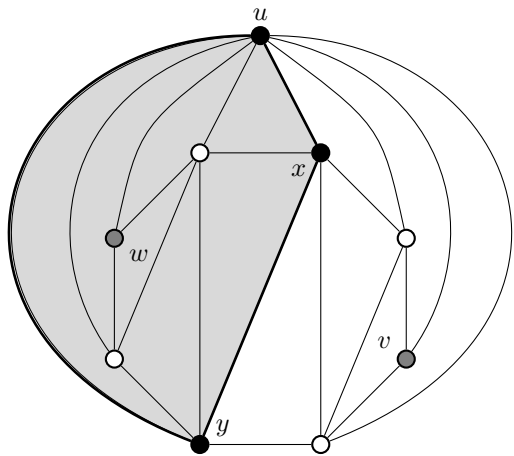
\begin{figure}[h]
\centering
	\begin{tikzpicture}
	[scale = 0.7, inner sep=0.8mm, 
	vertex/.style={circle,thick,draw}, 
	thickedge/.style={line width=1pt}] 
	
	\draw[fill=black!15] (0,5)--({1*45 + 22.5}:3)--({5*45 + 22.5}:3) .. controls (-7,-0.5) and (-5,5) .. (0,5);
	
	\node[vertex, fill=white] (0) at ({0*45 + 22.5}:3) {};
	\node[vertex, fill=black] (1) at ({1*45 + 22.5}:3) [label=213:{$x$}] {};
	\node[vertex, fill=white] (2) at ({2*45 + 22.5}:3) {};
	\node[vertex, fill=black!50] (3) at ({3*45 + 22.5}:3) [label=-30:{$w$}] {};
	\node[vertex, fill=white] (4) at ({4*45 + 22.5}:3) {};
	\node[vertex, fill=black] (5) at ({5*45 + 22.5}:3) [label=10:{$y$}] {};
	\node[vertex, fill=white] (6) at ({6*45 + 22.5}:3) {};
	\node[vertex, fill=black!50] (7) at ({7*45 + 22.5}:3) [label=150:{$v$}] {};
	
	\node[vertex, fill=black] (a) at (0, 5) [label=90:{$u$}] {};
	
	\draw (0)--(1)--(2)--(3)--(4)--(5)--(6)--(7)--(0);
	
	\draw (1)--(a) (2)--(a);
	
	\draw (0) .. controls (2.3,3.3) .. (a);
	\draw (3) .. controls (-2.3,3.3) .. (a);
	
	\draw (7) .. controls (4.3,1) and (3.8,4.2) .. (a);
	\draw (4) .. controls (-4.3,1) and (-3.8,4.2) .. (a);
	
	\draw (6) .. controls (7,-0.5) and (5,5) .. (a);
	\draw (5) .. controls (-7,-0.5) and (-5,5) .. (a);
	
	\draw (0)--(6)--(1)--(5)--(2)--(4);
	
	\draw[thickedge] (a)--(1)--(5);
	\draw[thickedge] (5) .. controls (-7,-0.5) and (-5,5) .. (a);
	
	\end{tikzpicture}
	\caption{The graph $G_9$ used in the proof of Theorem \ref{thm_mpg_order_nine}. The vertex $v$ is further from any vertex of a shaded face than $u$ is, and $w$ is further from any vertex of an unshaded face than $u$ is.}
	\label{fig_mpg_order_nine}
\end{figure}
	
	If $f$ is any shaded face in the interior of the bold cycle with vertex set $\{u,x,y\}$, then for all vertices $t$ in the boundary of $f$, we have that $d(v,t) > d(u,t)$. 
	If $f$ is an unshaded face in the exterior of the bold cycle (including the outer face of $G_9$), then all vertices $t$ in the boundary of $f$ satisfy $d(w,t) > d(u,t)$.
	Thus the vertex $u$ is not quasi-eccentric to any face, so $G_9$ is not contained in the center of any mpg.
\end{proof}

Using Corollary \ref{cor_centre_iff_equi}, we derive the following corollary of Theorems \ref{thm_order_eight} and \ref{thm_mpg_order_nine}.

\begin{cor}
	Every maximal planar graph of of order 8 or less is the center of some planar graph, and this bound is sharp.
	\label{cor_order_8_center}
\end{cor}

\section{Further questions}

Corollary \ref{cor_cen_char} gives a sufficient and necessary condition for a maximal planar graph to be the center of a planar graph. 
Does this result hold for a larger class of graphs than just maximal planar graphs?
We present the following conjecture as a possible answer to this question:

\begin{conj} 
	Let $H$ be a 2-connected plane graph such that each vertex of $H$ is quasi-eccentric to $H[f]$ for some face $f$ of $H$.
	There exists a plane graph $G$ such that $H$ is both the center of $G$ and an isometric subgraph of $G$.
\end{conj}

Figure \ref{fig:disccenter} demonstrates that the center of a maximal planar graph may be disconnected.
What conditions are sufficient and / or necessary for a (maximal) planar graph to have a connected center?

\section*{Acknowledgments}
Special thanks to David Erwin and Peter Dankelmann for valuable input.
Thank-you to the NRF for funding this research, grant number 120104, and to the University of Cape Town for their support.
I am grateful for the comments of an anonymous reviewer whose feedback contributed to improving this paper.

\bibliographystyle{amsplain}
\bibliography{GraphTheory}{}
	
\end{document}